\documentclass[11pt,reqno]{amsart}

\usepackage[margin=1.15in]{geometry}
\usepackage{amsmath,amssymb,amsthm,mathtools}
\usepackage{aliascnt}
\usepackage{bm}
\usepackage{enumitem}
\usepackage{microtype}
\usepackage[hidelinks]{hyperref}
\usepackage[nameinlink,capitalise,noabbrev]{cleveref}
\usepackage{tikz}
\usetikzlibrary{arrows.meta,positioning,calc,automata}

\hypersetup{
  pdftitle={Poisson occupancy of paths on permuted regular graphs},
  pdfauthor={Rafael Chiclana},
  pdfsubject={Random products of Markov kernels and quantitative Poisson approximation},
  pdfkeywords={random Markov kernels, random permutations, Poisson approximation, cutoff, averaging process}
}

\allowdisplaybreaks
\newtheorem{theorem}{Theorem}[section]

\newaliascnt{proposition}{theorem}

\aliascntresetthe{proposition}

\newaliascnt{lemma}{theorem}
\newtheorem{lemma}[lemma]{Lemma}
\aliascntresetthe{lemma}

\newaliascnt{corollary}{theorem}
\newtheorem{corollary}[corollary]{Corollary}
\aliascntresetthe{corollary}

\newaliascnt{claim}{theorem}

\aliascntresetthe{claim}

\theoremstyle{definition}
\newaliascnt{definition}{theorem}

\aliascntresetthe{definition}

\newaliascnt{example}{theorem}

\aliascntresetthe{example}

\theoremstyle{remark}
\newaliascnt{remark}{theorem}

\aliascntresetthe{remark}

\newcommand*{\Var}{\operatorname{Var}}
\newcommand{\one}{\mathbf{1}}
\newcommand{\E}{\mathbb{E}}
\newcommand{\Pp}{\mathbb{P}}
\newcommand{\C}{\mathbb{C}}

\newcommand{\TV}{\mathrm{TV}}
\newcommand{\Pois}{\operatorname{Pois}}
\newcommand{\Bin}{\operatorname{Bin}}

\newcommand{\tr}{\operatorname{tr}}

\newcommand{\ind}{\mathbf{1}}
\newcommand{\dd}{\,\mathrm{d}}

\newcommand{\T}{\mathsf{T}}

\title{Poisson occupancy of paths on permuted regular graphs}

\author{Rafael Chiclana}
\address{Center of Interdisciplinary Data Science and
	Artificial Intelligence (CIDSAI),
	New York University Abu Dhabi,
	Abu Dhabi, United Arab Emirates}

\email{rcv9691@nyu.edu}
\date{}


\begin{document}

\begin{abstract}
	In the classical occupancy problem, $B$ balls are placed independently
	and uniformly into $d$ boxes. We study a graph analogue of this model. Let $G$ be
	a directed $r$-regular graph on $d$ vertices and fix a vertex $i$.
	Consider all paths of length $L$ starting from $i$, and let $N_L(i,j)$
	denote the number of such paths ending at a vertex $j$. Thus the vertices
	play the role of boxes and the $r^L$ paths play the role of balls.
	
	In general, these occupancies depend strongly on the geometry of $G$. We show that inserting independent uniform permutations between
	successive layers restores the classical Poisson behavior at the
	empirical level. Although the path endpoints remain strongly dependent,
	their empirical occupancy distribution has the same limiting Poisson law
	as independent balls thrown into boxes.
	
	More precisely, at the critical scale $r^L/d\to\lambda\in(0,\infty)$, provided $rL/d \to 0$, the empirical path multiplicity distribution converges in Wasserstein
	distance to $\Pois(\lambda)$, independently of the geometry of the
	underlying graphs. We also prove a local Poisson
	limit theorem and a Berry--Esseen bound for the empirical distribution
	functions. These estimates determine the full total-variation transition
	of the normalized path multiplicities toward uniformity.
	
	For arbitrary doubly stochastic kernels, we further establish exact
	covariance and Frobenius norm identities, together with an exact contraction formula for the chi-square divergence. As applications, we study synchronous opinion averaging, obtaining an explicit cutoff profile for equal-size groups and quantitative convergence to consensus for heterogeneous group sizes.
\end{abstract}

\maketitle

\section{Introduction}

Consider the classical occupancy problem in which $B$ balls are thrown
independently and uniformly into $d$ boxes. If $N(j)$ denotes the
number of balls in box $j$, then
\[
N(j)\sim\Bin(B,1/d).
\]
Let $B=B(d)$ and assume that $B/d\to\lambda\in(0,\infty)$. Then the Poisson limit theorem gives
\[
N(j)\Rightarrow \Pois(\lambda).
\]
Rather than focusing on a fixed box, consider the random empirical
occupancy measure
\[
\nu_{d}
=
\frac1d\sum_{j=1}^d\delta_{N(j)}.
\]
The mass $\nu_{d}(\{k\})$ is the proportion of boxes containing exactly
$k$ balls. A classical law of large numbers for the occupancy problem
shows that, if $B/d\to\lambda\in(0,\infty)$, then the empirical measure converges to the
Poisson law. More precisely, 
\[
\E W_1\bigl(\nu_d,\Pois(\lambda)\bigr)\longrightarrow0,
\]
where $W_1$ denotes the  $1$-Wasserstein distance.

In this work we consider a graph analogue of this occupancy problem.
Let $G=(V,E)$ be an $r$-regular graph on $d$ vertices and fix $i\in V$.
Consider all paths of length $L$ starting from $i$. Note that there are exactly
\[
B=r^L
\]
such paths, and define $N_L(i,j)$ to be the number of them ending at vertex $j$. In this setting, vertices play the role of boxes and paths play the role of balls. Accordingly, one may consider the empirical measure
\begin{equation}\label{eq:empirical_measure}
\nu_L^{(i)}
=
\frac1d\sum_{j=1}^d\delta_{N_L(i,j)}.
\end{equation}

Consider the critical regime $B/d\to\lambda\in(0,\infty)$. In general, there is no reason for this measure to have a Poisson profile. For instance, if the connected component of $i$ has $o(d)$ vertices, then every path remains inside that component, and hence
\[
\nu_L^{(i)}(\{0\})\longrightarrow1.
\]
Thus the endpoint occupancies may depend strongly on the geometry of the graph $G$. To remove this geometric obstruction, we insert independent uniform
random permutations between successive layers. More precisely, before
each step the current vertices are permuted uniformly at random, after
which every path follows one edge of $G$. Crucially, the fresh permutations do not make the $r^L$ path endpoints
independent. At time $t-1$, the paths are grouped into $d$ piles of
sizes
\[
N_{t-1}(i,1),\ldots,N_{t-1}(i,d).
\]
The permutation $\pi_t$ does not act independently on the individual
paths. Instead, it randomly relabels these entire piles. In particular, all paths
currently occupying the same vertex remain together after the
permutation. The next layer does not break these piles apart either. Each
vertex $j$ receives the $r$ relabelled piles corresponding to its
incoming neighbors and adds their sizes. Thus the new multiplicities are obtained by shuffling and recombining
the old piles, rather than by independently reallocating the individual
paths. Moreover, the same permutation $\pi_t$ determines this
recombination for every vertex simultaneously, so the resulting pile
sizes are jointly dependent. Since the old pile sizes already encode
the randomness of all previous layers, this dependence is carried
through the entire evolution. See Figure \ref{fig:path-piles} for an illustration.

Nevertheless, our main result shows that this dependence does not affect
the limiting empirical occupancy law. Regardless of the geometry of
$G$, the endpoint multiplicities exhibit the same Poisson profile as in
the classical occupancy problem.

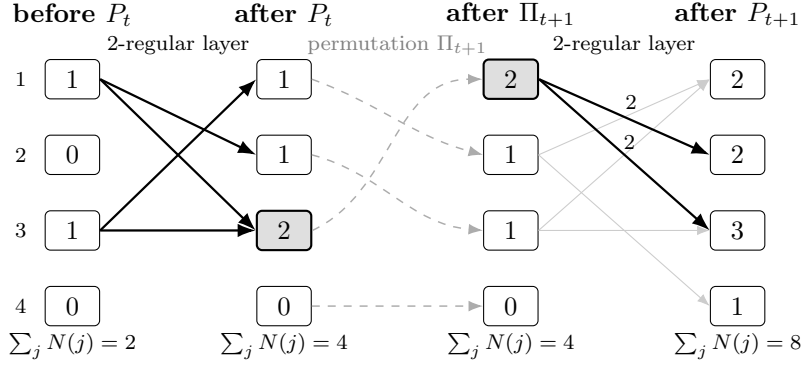
\begin{figure}[t]
	\centering
	\begin{tikzpicture}[
		>=Latex,
		every node/.style={font=\small},
		pile/.style={
			draw,
			rounded corners=2pt,
			minimum width=.72cm,
			minimum height=.52cm,
			inner sep=1pt,
			fill=white
		},
		highlighted/.style={
			pile,
			fill=gray!25,
			line width=.8pt
		},
		edge/.style={
			->,
			gray!40,
			line width=.4pt
		},
		active/.style={
			->,
			black,
			line width=.85pt
		},
		perm/.style={
			->,
			dashed,
			gray!65,
			line width=.55pt
		}
		]
		
		\def\xA{0}
		\def\xB{2.8}
		\def\xC{5.8}
		\def\xD{8.8}
		
		\def\yA{1.5}
		\def\yB{0.5}
		\def\yC{-0.5}
		\def\yD{-1.5}
		
		\node[font=\small\bfseries] at (\xA,2.35)
		{before $P_t$};
		\node[font=\small\bfseries] at (\xB,2.35)
		{after $P_t$};
		\node[font=\small\bfseries] at (\xC,2.35)
		{after $\Pi_{t+1}$};
		\node[font=\small\bfseries] at (\xD,2.35)
		{after $P_{t+1}$};
		
		\node[pile] (A1) at (\xA,\yA) {$1$};
		\node[pile] (A2) at (\xA,\yB) {$0$};
		\node[pile] (A3) at (\xA,\yC) {$1$};
		\node[pile] (A4) at (\xA,\yD) {$0$};
		
		\node[left=3pt of A1,font=\scriptsize] {$1$};
		\node[left=3pt of A2,font=\scriptsize] {$2$};
		\node[left=3pt of A3,font=\scriptsize] {$3$};
		\node[left=3pt of A4,font=\scriptsize] {$4$};
		
		%
		%
		\node[pile]        (B1) at (\xB,\yA) {$1$};
		\node[pile]        (B2) at (\xB,\yB) {$1$};
		\node[highlighted] (B3) at (\xB,\yC) {$2$};
		\node[pile]        (B4) at (\xB,\yD) {$0$};
		
		\draw[active] (A1.east) -- (B2.west);
		\draw[active] (A1.east) -- (B3.west);
		
		\draw[active] (A3.east) -- (B1.west);
		\draw[active] (A3.east) -- (B3.west);
		
		\node[font=\scriptsize] at (1.40,1.95)
		{$2$-regular layer};
		
		%
		\node[highlighted] (C1) at (\xC,\yA) {$2$};
		\node[pile]        (C2) at (\xC,\yB) {$1$};
		\node[pile]        (C3) at (\xC,\yC) {$1$};
		\node[pile]        (C4) at (\xC,\yD) {$0$};
		
		\draw[perm] (B3.east) to[out=25,in=180] (C1.west);
		\draw[perm] (B1.east) to[out=-20,in=165] (C2.west);
		\draw[perm] (B2.east) to[out=-15,in=175] (C3.west);
		\draw[perm] (B4.east) -- (C4.west);
		
		\node[font=\scriptsize,gray!100] at (4.30,1.95)
		{permutation $\Pi_{t+1}$};
		
		%
		%
		\node[pile] (D1) at (\xD,\yA) {$2$};
		\node[pile] (D2) at (\xD,\yB) {$2$};
		\node[pile] (D3) at (\xD,\yC) {$3$};
		\node[pile] (D4) at (\xD,\yD) {$1$};
		
		\draw[edge] (C2.east) -- (D1.west);
		\draw[edge] (C2.east) -- (D4.west);
		
		\draw[edge] (C3.east) -- (D1.west);
		\draw[edge] (C3.east) -- (D3.west);
		
		\draw[active]
		(C1.east) -- node[pos=.54,above,font=\scriptsize] {$2$} (D2.west);
		\draw[active]
		(C1.east) -- node[pos=.53,above,font=\scriptsize] {$2$} (D3.west);
		
		\node[font=\scriptsize] at (7.30,1.95)
		{$2$-regular layer};
		
		\node[font=\scriptsize] at (\xA,-2.03) {$\sum_j N(j)=2$};
		\node[font=\scriptsize] at (\xB,-2.03) {$\sum_j N(j)=4$};
		\node[font=\scriptsize] at (\xC,-2.03) {$\sum_j N(j)=4$};
		\node[font=\scriptsize] at (\xD,-2.03) {$\sum_j N(j)=8$};
		
	\end{tikzpicture}
	
	\caption{
		Illustration of two successive layers of the permuted graph occupancy model $R_L$ for a $2$-regular graph with four vertices.
	}
	\label{fig:path-piles}
\end{figure}

More generally, we allow both the graph and its degree to vary between
layers. Let $P_t$ be the transition matrix of an $r_t$-regular graph on
$[d]$, and let $\Pi_1,\ldots,\Pi_L$ be independent uniform permutation
matrices. We consider
\begin{equation}\label{eq:regular-product}
	R_L=\Pi_1P_1\Pi_2P_2\cdots\Pi_LP_L.
\end{equation}
For this model, the total number of paths of length $L$ is
\begin{equation}\label{eq:B_L-definition}
	B_L=\prod_{t=1}^L r_t.
\end{equation}
The same path-count interpretation holds: $N_L(i,j)$ denotes the number
of paths of length $L$ from $i$ to $j$, and the empirical multiplicity measure $\nu_L^{(i)}$ is defined as in \eqref{eq:empirical_measure}. We set
\[
\lambda_L=\frac{B_L}{d},
\qquad
S_L=\frac1d\sum_{t=1}^L r_t.
\]
Our main result gives a quantitative bound on the $1$-Wasserstein distance between
$\nu_L^{(i)}$ and $\Pois(\lambda_L)$. Notably, the bound depends only
on the degree sequence $(r_t)_{t=1}^L$ and is independent of the geometry of the
graphs.

\begin{theorem}\label{cor:wasserstein-poisson}
	Assume $S_L\leq1/4$. Then, for every $i\in[d]$,
	\begin{equation}\label{eq:wasserstein-bound-main}
		\E W_1\bigl(\nu_L^{(i)},\Pois(\lambda_L)\bigr)
		\leq
		14\lambda_L\sqrt{1+\lambda_L}
		\sqrt{
			S_L\frac{\log(1/S_L)}
			{\log\log(1/S_L)}
		}.
	\end{equation}
	In particular, if $\lambda_L\to\lambda\in[0,\infty)$ and $S_L\to0$, then $W_1\bigl(\nu_L^{(i)},\Pois(\lambda)\bigr)\to0$ in $L^1$.
\end{theorem}

The condition $S_L\to0$ cannot in general be omitted. Indeed, if
$L=1$ and $r_1/d\to\lambda\in(0,1)$, then
\[
\nu_1^{(i)}
=
\left(1-\frac{r_1}{d}\right)\delta_0
+
\frac{r_1}{d}\delta_1,
\]
which converges to a Bernoulli law rather than $\Pois(\lambda)$.

We also obtain sharper estimates for point probabilities and distribution functions. The first is a local Poisson approximation, analogous to a local central limit theorem.

\begin{theorem}\label{cor:local-poisson}
	Assume $S_L\leq 1/2$. For every $i\in[d]$,
	\begin{equation}\label{eq:local-bound-main}
		\E\sup_{k\geq0}
		\left|
		\nu_L^{(i)}(\{k\})
		-
		\Pp\bigl(Z_{\lambda_L}=k\bigr)
		\right|
		\leq
		\frac{12\sqrt{2}}{\pi}\lambda_L\sqrt{S_L},
	\end{equation}
	where $Z_{\lambda_L}\sim\Pois(\lambda_L)$.
\end{theorem}

We also obtain a Berry--Esseen bound for the distribution functions.

\begin{theorem}\label{cor:poisson-berry-esseen}
	Assume $S_L\leq 1/2$. For every $i\in[d]$,
	\begin{equation}\label{eq:kolmogorov-bound-main}
		\E\sup_{m\geq0}
		\left|
		\nu_L^{(i)}([0,m])
		-
		\Pp\bigl(Z_{\lambda_L}\leq m\bigr)
		\right|
		\leq
		\sqrt{18}\lambda_L\sqrt{S_L},
	\end{equation}
	where $Z_{\lambda_L}\sim\Pois(\lambda_L)$.
\end{theorem}

The numerical constants in Theorems~\ref{cor:wasserstein-poisson}--\ref{cor:poisson-berry-esseen}
are not optimized. Our aim is to make the dependence on the parameters
explicit while keeping the argument transparent.

The Poisson approximation above can be used to describe the mixing
profile of $R_L$. We first note that
\[
(R_L)_{ij}=\frac{N_L(i,j)}{B_L}.
\]
Thus, the $i$-th row of $R_L$, denoted $(R_L)_i$, records how the total path mass is distributed among the vertices when the starting vertex is $i$. The uniform distribution
\[
u=\frac1d(1,\ldots,1)
\]
is stationary under every factor of the product and corresponds to the
equilibrium state in which the path mass is spread equally among all
vertices. The empirical multiplicity measure allows us to track the
distance from this equilibrium. Indeed,
\[
\bigl\|(R_L)_i-u\bigr\|_{\TV}
=
\frac{1}{2\lambda_L}
\int |k-\lambda_L|\,\dd\nu_L^{(i)}(k).
\]
In other words, the total-variation distance is a Lipschitz functional of the empirical multiplicity measure. Thus, in the critical regime $B_L\asymp d$,
Wasserstein convergence yields the limiting total-variation
profile. Before this scale, a deterministic support obstruction keeps
the rows far from uniform, while beyond it an exact $\chi^2$ contraction
estimate drives them to equilibrium. Together, these give the full
total-variation profile.

\begin{corollary}\label{thm:tv-profile}
	For every $i\in[d]$, the following hold.
	\begin{enumerate}[label=\textup{(\roman*)}]
		\item If $\lambda_L\leq 1$, then for every realization of $R_L$ we have
		\[
		\bigl\|(R_L)_i-u\bigr\|_{\TV}\geq 1-\lambda_L.
		\]
	\item If $S_L\leq1/2$, then
	\begin{equation}\label{eq:tv-profile-rate-main}
		\E\left|
		\bigl\|(R_L)_i-u\bigr\|_{\TV}
		-
		\Psi(\lambda_L)
		\right|
		\leq
		\sqrt{18}\lambda_L\sqrt{S_L},
	\end{equation}
	where
	\begin{equation}\label{eq:Psi-definition} 
		\Psi(\lambda)=\frac{1}{2\lambda} \E |Z_\lambda - \lambda|,\qquad Z_\lambda \sim \Pois(\lambda).
	\end{equation}
	Consequently, if $\lambda_L\to\lambda\in(0,\infty)$ and $S_L\to0$, then $\bigl\|(R_L)_i-u\bigr\|_{\TV}
	\to
	\Psi(\lambda)$
	in $L^1$.
		
		\item For every $L$,
		\[
		\E\bigl\|(R_L)_i-u\bigr\|_{\TV}
		\leq
		\frac{1}{2\sqrt{\lambda_L}}.
		\]
		In particular, if $\lambda_L\to\infty$, then $\bigl\|(R_L)_i-u\bigr\|_{\TV}\to 0$ in $L^1$.
	\end{enumerate}
\end{corollary}

The limiting profile admits the explicit formula
\[
\Psi(\lambda)
=
\Pp\bigl(Z_\lambda=\lfloor\lambda\rfloor\bigr)
=
e^{-\lambda}
\frac{\lambda^{\lfloor\lambda\rfloor}}
{\lfloor\lambda\rfloor!}.
\]
In particular, $\Psi(\lambda)\to1$ as $\lambda\downarrow0$, while
Stirling's formula gives
\[
\Psi(\lambda)\sim\frac{1}{\sqrt{2\pi\lambda}}
\qquad\text{as }\lambda\to\infty.
\]

\subsection{Main ingredients}

The technical core behind the Poisson approximation is a uniform estimate for the empirical probability generating function. For
\[
M_L^{(i)}(z)
=
\int z^k\,\dd\nu_L^{(i)}(k),
\qquad |z|\leq1,
\]
we prove the following.

\begin{theorem}\label{thm:complex-pgf}
	Assume $S_L\leq1/2$. Then, for every $i\in[d]$ and every $z\in\overline{\mathbb D}$,
	\begin{equation}\label{eq:poisson-pgf-main}
		\E\left|
		M_L^{(i)}(z)-e^{\lambda_L(z-1)}
		\right|^2
		\leq
		18\lambda_L^2|1-z|^2S_L.
	\end{equation}
\end{theorem}

We note that the probability generating function of $\Pois(\lambda_L)$ is
exactly $e^{\lambda_L(z-1)}$. Thus Theorem~\ref{thm:complex-pgf} gives a quantitative $L^2$ bound for the deviation of the empirical generating function from the Poisson generating function. This is the main analytic input behind our main results \Cref{cor:wasserstein-poisson,cor:local-poisson,cor:poisson-berry-esseen}. The factor $|1-z|^2$ is particularly important, as it cancels the denominator
that appears when cumulative probabilities are recovered from generating
functions, allowing us to control distribution functions by Fourier
analysis.

The postcritical estimate in Corollary~\ref{thm:tv-profile} is obtained
from a different ingredient: an exact covariance identity for products
of arbitrary doubly stochastic matrices interlaced with independent
uniform permutations. 

Let $P_1,\ldots,P_L \in \mathbb{R}^{d\times d}$ be arbitrary doubly stochastic matrices and consider the product $R_L$ as in \eqref{eq:regular-product}. For $t=1,\ldots,L$, set
\[
J=\frac{1}{d}\one\one^\top,
\qquad
Q_t=P_t-J,
\qquad
\alpha(P_t)=\frac{\|Q_t\|_F^2}{d-1}.
\]
Then $0\leq\alpha(P_t)\leq1$. Averaging over the independent uniform
permutations makes the centered covariance isotropic, while each layer
contributes only through the scalar factor $\alpha(P_t)$. This yields the following exact contraction formula.

\begin{theorem}\label{thm:exact-covariance}
	Let $P_1,\ldots,P_L$ be arbitrary doubly stochastic kernels on $[d]$,
	and let $R_L$ be the product defined in \eqref{eq:regular-product}, where
	$\Pi_1,\ldots,\Pi_L$ are independent uniform permutation matrices. Then
	\begin{equation}\label{eq:left-covariance}
		\E\bigl[(R_L-J)(R_L-J)^\top\bigr]
		=
		\left(\prod_{t=1}^L\alpha(P_t)\right)
		(I-J).
	\end{equation}
\end{theorem}

As an immediate consequence, we obtain exact formulas for the expected Frobenius distance from \(J\) and for the expected \(\chi^2\)-divergence from equilibrium.

\begin{corollary}\label{cor:exact-l2}
	Under the assumptions of Theorem~\ref{thm:exact-covariance},
	\begin{equation}\label{eq:exact-frobenius}
		\E\|R_L-J\|_F^2
		=
		(d-1)\prod_{t=1}^L\alpha(P_t).
	\end{equation}
	Moreover, for every probability row vector $\mu$ on $[d]$,
	\begin{equation}\label{eq:exact-chi-square-general}
		\E\chi^2\bigl(\mu R_L\,\|\,u\bigr)
		=
		\chi^2(\mu\|u)\prod_{t=1}^L\alpha(P_t).
	\end{equation}
\end{corollary}

We believe that the exact identities in Theorem~\ref{thm:exact-covariance} and Corollary~\ref{cor:exact-l2} are of independent interest. As a concrete application, we use them in Section~\ref{sec:applications} to establish quantitative convergence to consensus for
synchronous opinion averaging with heterogeneous room sizes.

Using the standard bound
\begin{equation}\label{eq:standard_divergence}
\|\mu-\nu\|_{\TV}\leq \frac12\sqrt{\chi^2(\mu\|\nu)},
\end{equation}
together with Jensen's inequality, \eqref{eq:exact-chi-square-general} yields
\[
\E\bigl\|\mu R_L-u\bigr\|_{\TV}
\leq
\frac12\sqrt{\chi^2(\mu\|u)}
\left(\prod_{t=1}^L\alpha(P_t)\right)^{1/2}.
\]
Consequently, if $\chi^2(\mu\|u)\prod_{t=1}^L\alpha(P_t)\to 0$, then $\mu R_L$ converges to $u$ in total variation in $L^1$. For an $r$-regular kernel,
\[
\alpha(P)
=
\frac{d-r}{r(d-1)}
\leq \frac1r.
\]
Consequently, \eqref{eq:exact-chi-square-general} gives
\[
\E\chi^2\bigl((R_L)_i\|u\bigr)
\leq
\frac{d-1}{B_L}
\leq
\frac{1}{\lambda_L}.
\]
Together with (\ref{eq:standard_divergence}), this yields the postcritical estimate in Corollary~\ref{thm:tv-profile}.

\subsection{Applications}\label{subsec:applications}

We illustrate our results with a synchronous model of opinion averaging. Consider a
population of $d$ agents, each holding a nonnegative real-valued opinion, and fix an
integer $b\geq2$ dividing $d$. At each round, the agents are divided
uniformly at random into groups of size $b$ and sent to separate rooms.
The agents in each room talk to one another and, at the end of the round,
all replace their opinions by the average opinion in that room. The
agents then leave the rooms, are reshuffled independently into new
groups of size $b$, and the procedure is repeated. See Figure \ref{fig:random-rooms} for an illustration.

Since every update preserves the average opinion of the population, the
natural limiting state is consensus at the initial population average.
It is therefore natural to ask how quickly the system approaches
consensus as the agents repeatedly meet and reshuffle.

Up to a final permutation of the agents, this dynamics is obtained from the model (\ref{eq:regular-product}) by setting $P_1=\ldots=P_L=P$, where $P$ is the transition matrix of the graph $G$ consisting of $d/b$ disconnected components, each a complete graph of size $b$, with self-loops. The vertices represent the agents and each connected component represents one room. The action of $P$ replaces the value at each vertex by the average over its connected component. Finally, randomly assigning the agents to rooms at each round corresponds to applying independent uniform permutations to the vertices. In particular, all layers have degree $b$, so
\[
B_L=b^L,
\qquad
S_L=\frac{bL}{d},
\qquad
\lambda_L=\frac{b^L}{d}.
\]
The path multiplicities also have a direct interpretation. If we follow
the initial influence of an agent $i$, after $L$ rounds it has been split
into $b^L$ equal pieces, and $N_L(i,j)$ is the number of these pieces
received by agent $j$. Equivalently, up to a final permutation of the agents, the $i$-th row of the product $R_L$ defined in (\ref{eq:regular-product}), whose entries are given by
\[
(R_L)_{ij}=\frac{N_L(i,j)}{b^L},
\]
corresponds to the influence profile of agent $i$ across the population. The empirical measure $\nu_L^{(i)}$ describes the distribution of the influence of agent $i$. More precisely, $\nu_L^{(i)}(\{k\})$ is the proportion of agents receiving exactly $k$ of these pieces. 

Our general results give a particularly simple description of this dynamics. They determine both how the influence of a single agent is distributed across the population and how the system approaches consensus. We record these consequences without the quantitative error bounds, which follow directly from our general results stated above.

Normalize the initial opinion profile so that it is a probability measure $\mu$ on $[d]$. Consensus corresponds to the uniform measure $u$ on $[d]$. We measure worst-case convergence to consensus by
\begin{equation}\label{eq:consensus}
\mathcal C_d(L)
=
\sup_{\mu\in\mathcal P([d])}
\E\|\mu R_L-u\|_{\TV},
\end{equation}
where the supremum is over all probability vectors $\mu$ on $[d]$. Recall that $\Psi(\lambda)=\frac{1}{2\lambda} \E |Z_\lambda - \lambda|$, where $Z_\lambda \sim \Pois(\lambda)$.

\begin{theorem}\label{thm:synchronous-averaging}
	Let $b=b(d)$ divide $d$, and suppose that, as $d \to \infty$,
	\[
	\frac{bL}{d}\longrightarrow0,
	\qquad
	\frac{b^L}{d}\longrightarrow\lambda\in(0,\infty).
	\]
	Then, for every fixed agent $i$,
	\[
	\E W_1\bigl(\nu_L^{(i)},\Pois(\lambda)\bigr)\longrightarrow0,
	\]
	and
	\[
	\mathcal C_d(L)\longrightarrow\Psi(\lambda).
	\]
\end{theorem}

For fixed room size, the preceding theorem gives an explicit cutoff profile. To avoid irrelevant arithmetic effects, we restrict to $d=b^n$ for some $n\in \mathbb{N}$.

\begin{corollary}\label{cor:synchronous-cutoff}
	Fix $b\geq2$ and let $d=b^n$. Then, for every fixed
	$c\in\mathbb Z$,
	\[
	\mathcal C_{d}(n+c)
	\longrightarrow
	\Psi(b^c)
	\qquad\text{as }n\to\infty.
	\]
	Consequently, the synchronous opinion averaging process exhibits cutoff for
	convergence to consensus at $\log_b d$, with a window of order one and
	discrete cutoff profile $c\mapsto\Psi(b^c)$.
\end{corollary}

Finally, we consider a heterogeneous version of this model that allows the rooms to have different sizes. At round $t$, fix
integers
\[
b_{t,1},\ldots,b_{t,q_t}\geq1,
\qquad
\sum_{a=1}^{q_t} b_{t,a}=d,
\]
and partition the agents uniformly at random into $q_t$ rooms of these
sizes. As before, all agents in a room simultaneously replace their
opinions by the average opinion in that room. See Figure \ref{fig:random-rooms} for an illustration.

This model also fits into our framework. At round $t$, let $P_t$ be the transition matrix of the graph
consisting of $q_t$ disconnected complete components of sizes
$b_{t,1},\ldots,b_{t,q_t}$, with self-loops. Each component represents
one room, and the action of $P_t$ replaces the opinion of every agent by
the average opinion in its room. Randomly assigning the agents to these
rooms corresponds, as before, to applying a uniform random permutation
to the vertices.

When the room sizes are not all equal, $P_t$ is no longer regular. Nevertheless, $P_t$ is doubly stochastic, so the exact $L^2$ identities of \cref{thm:exact-covariance,cor:exact-l2}
still apply. As we show in Section~\ref{sec:applications}, they yield an exact formula for the expected $\chi^2$-divergence from consensus which, surprisingly, depends only on the number $q_t$ of rooms at each round and not on their individual sizes. In particular, they provide a quantitative bound on convergence to consensus in total variation.

\begin{theorem}\label{thm:heterogeneous-averaging}
	Consider the synchronous opinion averaging process in which, at round
	$t$, the agents are divided into $q_t$ rooms of arbitrary sizes. Then
	\[
	\mathcal C_d(L)
	\leq
	\frac12
	\sqrt{
		d\prod_{t=1}^L\frac{q_t}{d}
	}.
	\]
\end{theorem}

\begin{figure}[t]
	\centering
	\begin{tikzpicture}[
		>=Latex,
		every node/.style={font=\small},
		agent/.style={
			circle,
			draw,
			fill=white,
			minimum size=5pt,
			inner sep=0pt
		},
		room/.style={
			draw,
			rounded corners=4pt,
			fill=gray!5
		},
		connection/.style={
			gray!45,
			line width=.35pt
		},
		flow/.style={
			->,
			line width=.7pt
		}
		]
		
		
		\begin{scope}[xshift=0cm]
			
			\node[font=\small\bfseries] at (3.0,3.55)
			{Homogeneous rooms};
			
			\foreach \x in {0.75,1.65,2.55,3.45,4.35,5.25} {
				\node[agent] at (\x,2.95) {};
				\node[agent] at (\x,2.60) {};
			}

			\draw[flow] (3.0,2.10) -- (3.0,1.55);
			
			\node[font=\scriptsize,anchor=west] at (3.08,1.83)
			{uniform reshuffling};
			
			
			\node[room,minimum width=1.55cm,minimum height=1.20cm]
			at (1.0,0.55) {};
			
			\node[agent] (L11) at (0.72,0.82) {};
			\node[agent] (L12) at (1.28,0.82) {};
			\node[agent] (L13) at (0.72,0.28) {};
			\node[agent] (L14) at (1.28,0.28) {};
			
			\draw[connection] (L11)--(L12);
			\draw[connection] (L11)--(L13);
			\draw[connection] (L11)--(L14);
			\draw[connection] (L12)--(L13);
			\draw[connection] (L12)--(L14);
			\draw[connection] (L13)--(L14);
			
			\draw[connection]
			(L11) .. controls +(120:.30) and +(60:.30) .. (L11);
			\draw[connection]
			(L12) .. controls +(120:.30) and +(60:.30) .. (L12);
			\draw[connection]
			(L13) .. controls +(240:.30) and +(300:.30) .. (L13);
			\draw[connection]
			(L14) .. controls +(240:.30) and +(300:.30) .. (L14);
			
			
			\node[room,minimum width=1.55cm,minimum height=1.20cm]
			at (3.0,0.55) {};
			
			\node[agent] (L21) at (2.72,0.82) {};
			\node[agent] (L22) at (3.28,0.82) {};
			\node[agent] (L23) at (2.72,0.28) {};
			\node[agent] (L24) at (3.28,0.28) {};
			
			\draw[connection] (L21)--(L22);
			\draw[connection] (L21)--(L23);
			\draw[connection] (L21)--(L24);
			\draw[connection] (L22)--(L23);
			\draw[connection] (L22)--(L24);
			\draw[connection] (L23)--(L24);
			
			\draw[connection]
			(L21) .. controls +(120:.30) and +(60:.30) .. (L21);
			\draw[connection]
			(L22) .. controls +(120:.30) and +(60:.30) .. (L22);
			\draw[connection]
			(L23) .. controls +(240:.30) and +(300:.30) .. (L23);
			\draw[connection]
			(L24) .. controls +(240:.30) and +(300:.30) .. (L24);
			
			
			\node[room,minimum width=1.55cm,minimum height=1.20cm]
			at (5.0,0.55) {};
			
			\node[agent] (L31) at (4.72,0.82) {};
			\node[agent] (L32) at (5.28,0.82) {};
			\node[agent] (L33) at (4.72,0.28) {};
			\node[agent] (L34) at (5.28,0.28) {};
			
			\draw[connection] (L31)--(L32);
			\draw[connection] (L31)--(L33);
			\draw[connection] (L31)--(L34);
			\draw[connection] (L32)--(L33);
			\draw[connection] (L32)--(L34);
			\draw[connection] (L33)--(L34);
			
			\draw[connection]
			(L31) .. controls +(120:.30) and +(60:.30) .. (L31);
			\draw[connection]
			(L32) .. controls +(120:.30) and +(60:.30) .. (L32);
			\draw[connection]
			(L33) .. controls +(240:.30) and +(300:.30) .. (L33);
			\draw[connection]
			(L34) .. controls +(240:.30) and +(300:.30) .. (L34);
			
			\node[font=\scriptsize] at (1.0,-0.2) {$4$};
			\node[font=\scriptsize] at (3.0,-0.2) {$4$};
			\node[font=\scriptsize] at (5.0,-0.2) {$4$};
			
		\end{scope}

		
		\draw[gray!35] (6.25,-0.45) -- (6.25,3.75);

		
		\begin{scope}[xshift=6.65cm]
			
			\node[font=\small\bfseries] at (3.25,3.55)
			{Heterogeneous rooms};
			
			\foreach \x in {0.70,1.70,2.70,3.70,4.70,5.70} {
				\node[agent] at (\x,2.95) {};
				\node[agent] at (\x,2.60) {};
			}

			\draw[flow] (3.20,2.10) -- (3.20,1.55);
			
			\node[font=\scriptsize,anchor=west] at (3.28,1.83)
			{uniform reshuffling};
			
			
			\node[room,minimum width=1.10cm,minimum height=1.05cm]
			at (0.65,0.55) {};
			
			\node[agent] (R11) at (0.45,0.55) {};
			\node[agent] (R12) at (0.85,0.55) {};
			
			\draw[connection] (R11)--(R12);
			
			\draw[connection]
			(R11) .. controls +(150:.30) and +(210:.30) .. (R11);
			\draw[connection]
			(R12) .. controls +(30:.30) and +(-30:.30) .. (R12);
			
			\node[font=\scriptsize] at (0.65,-0.15) {$2$};
			
			
			\node[room,minimum width=1.40cm,minimum height=1.15cm]
			at (2.20,0.55) {};
			
			\node[agent] (R21) at (2.20,0.90) {};
			\node[agent] (R22) at (1.88,0.28) {};
			\node[agent] (R23) at (2.52,0.28) {};
			
			\draw[connection] (R21)--(R22);
			\draw[connection] (R21)--(R23);
			\draw[connection] (R22)--(R23);
			
			\draw[connection]
			(R21) .. controls +(120:.30) and +(60:.30) .. (R21);
			\draw[connection]
			(R22) .. controls +(180:.30) and +(240:.30) .. (R22);
			\draw[connection]
			(R23) .. controls +(0:.30) and +(-60:.30) .. (R23);
			
			\node[font=\scriptsize] at (2.20,-0.20) {$3$};
			
			
			\node[room,minimum width=3.10cm,minimum height=1.45cm]
			at (4.95,0.55) {};
			
			\node[agent] (R31) at (4.95,1.05) {};
			\node[agent] (R32) at (5.45,0.88) {};
			\node[agent] (R33) at (5.62,0.43) {};
			\node[agent] (R34) at (5.28,0.12) {};
			\node[agent] (R35) at (4.62,0.12) {};
			\node[agent] (R36) at (4.28,0.43) {};
			\node[agent] (R37) at (4.45,0.88) {};
			
			\draw[connection] (R31)--(R32);
			\draw[connection] (R31)--(R33);
			\draw[connection] (R31)--(R34);
			\draw[connection] (R31)--(R35);
			\draw[connection] (R31)--(R36);
			\draw[connection] (R31)--(R37);
			
			\draw[connection] (R32)--(R33);
			\draw[connection] (R32)--(R34);
			\draw[connection] (R32)--(R35);
			\draw[connection] (R32)--(R36);
			\draw[connection] (R32)--(R37);
			
			\draw[connection] (R33)--(R34);
			\draw[connection] (R33)--(R35);
			\draw[connection] (R33)--(R36);
			\draw[connection] (R33)--(R37);
			
			\draw[connection] (R34)--(R35);
			\draw[connection] (R34)--(R36);
			\draw[connection] (R34)--(R37);
			
			\draw[connection] (R35)--(R36);
			\draw[connection] (R35)--(R37);
			
			\draw[connection] (R36)--(R37);
			
			\draw[connection]
			(R31) .. controls +(120:.27) and +(60:.27) .. (R31);
			\draw[connection]
			(R32) .. controls +(60:.27) and +(0:.27) .. (R32);
			\draw[connection]
			(R33) .. controls +(30:.27) and +(-30:.27) .. (R33);
			\draw[connection]
			(R34) .. controls +(-60:.27) and +(-120:.27) .. (R34);
			\draw[connection]
			(R35) .. controls +(-120:.27) and +(-60:.27) .. (R35);
			\draw[connection]
			(R36) .. controls +(210:.27) and +(150:.27) .. (R36);
			\draw[connection]
			(R37) .. controls +(150:.27) and +(90:.27) .. (R37);
			
			\node[font=\scriptsize] at (4.95,-0.35) {$7$};
			
		\end{scope}
		
	\end{tikzpicture}
	
	\caption{Synchronous opinion averaging process with $12$ agents, with homogeneous and heterogeneous room sizes.}
	\label{fig:random-rooms}
\end{figure}
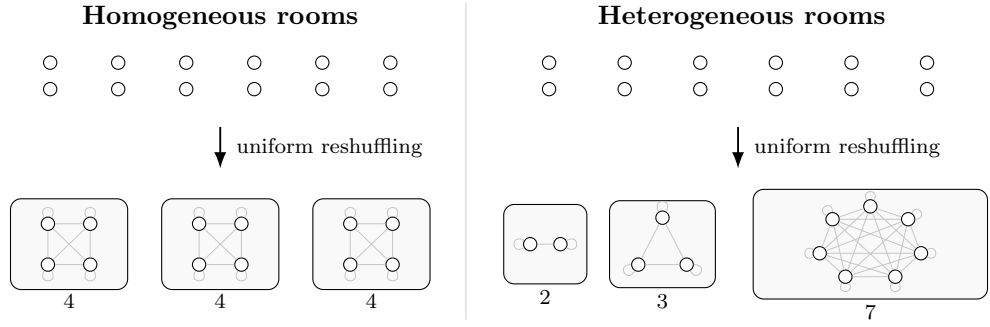

\subsection{Relation to previous work}

Poisson approximation in classical occupancy and allocation models is
well established; see, for example, the monograph \cite{BarbourHolstJanson1992}. For sums of dependent indicators, Arratia, Goldstein, and Gordon
\cite{ArratiaGoldsteinGordon1989} developed quantitative Chen--Stein
bounds in terms of local dependence parameters. Multivariate Poisson
approximation for multinomial and related allocation schemes was studied
in \cite{McDonald1980,Roos1999}. Poisson approximations for graph path multiplicities also
appear in the work of Lubetzky and Sly \cite{LubetzkySly2010} on non-backtracking random walks on random regular graphs. Their analysis exploits the typical geometry of random regular graphs, whereas our results apply to arbitrary deterministic regular layers, with independent uniform permutations resampled at every step.

Averaging processes are dynamics in which agents repeatedly replace their
values by local averages over randomly chosen interaction sets. They have attracted substantial attention in connection with interacting particle systems and opinion dynamics; see, for example, \cite{BoydGhoshPrabhakarShah2006,AldousLanoue2012} and the references
therein. In the pairwise case, Aldous and Lanoue \cite{AldousLanoue2012} related the
dynamics to an associated reversible Markov chain and obtained quantitative
convergence estimates in terms of its spectral gap. More recently,
Chatterjee, Diaconis, Sly, and Zhang \cite{ChatterjeeDiaconisSlyZhang2022} obtained a sharp phase transition for repeated pairwise averaging on the complete graph, and Movassagh,
Szegedy, and Wang \cite{MovassaghSzegedyWang2024} studied the process on
general graphs; see also
\cite{CaputoQuattropaniSau2023,CaputoQuattropaniSau2026} for further
mixing results on specific geometries. 

Particularly relevant is the work of Caputo, Quattropani, and Sau
\cite{CaputoQuattropaniSau2024} on repeated block averages. Their
dynamics is sequential: at each step a single random block of coordinates
is averaged. They obtain sharp criteria for cutoff, identify the cutoff
window and profile, and, for certain polynomially growing sublinear block sizes,
describe a regime with a Poissonian total-variation profile, described in terms of arrival times of a Poisson process. By contrast, in the synchronous model introduced above, all blocks of a random partition are averaged simultaneously. Under our sparsity assumption, we establish a Poisson law for the empirical path multiplicity distribution at the critical scale where the number of paths is comparable to the size of the system. This, in turn, yields an explicit total-variation profile.

A second closely related direction concerns the use of permutations to
accelerate Markov chains. Chatterjee and Diaconis
\cite{ChatterjeeDiaconis2020} showed that inserting suitable
deterministic permutations can substantially accelerate mixing.
Ben-Hamou and Peres \cite{BenHamouPeres2023} subsequently considered a
uniformly random permutation and proved cutoff at an entropic time under
general assumptions; related mixtures of permuted kernels were studied
by Dubail \cite{Dubail2024Reversible,Dubail2024General}. In these works
the random permutation is sampled and then reused throughout the
evolution. Dynamic random permutations were studied more recently by
Avena, van der Hofstad, den Hollander, and Nagy
\cite{AvenaVanDerHofstadDenHollanderNagy2026}, where the environment
itself evolves through random transpositions. In our model, the permutation is resampled independently and uniformly at each layer. This prevents geometric correlations from persisting across layers, while the paths within a given layer remain strongly dependent.

Finally, products of stochastic matrices are a standard framework for
studying consensus and time-inhomogeneous Markov chains; see, for example,
\cite{TouriNedic2011,TouriNedic2014,BajovicXavierMouraSinopoli2013}.
Touri and Nedi\'c developed general ergodicity and convergence criteria
for random stochastic matrix products, while Bajovi\'c, Xavier, Moura,
and Sinopoli identified exact asymptotic rates of convergence for certain
i.i.d.\ models. In our setting, for arbitrary deterministic doubly
stochastic layers, the centered covariance can be computed explicitly at
every finite time. This yields exact Frobenius and $\chi^2$ divergence contraction
identities.

\section{Notation and setup}\label{sec:two}

For $d\geq 2$ we write $[d]=\{1,\ldots,d\}$ and let $S_d$ denote the symmetric group on $[d]$. The open unit disk is denoted by
\[
\mathbb D=\{z\in\mathbb C:|z|<1\},
\]
and $\overline{\mathbb{D}}$ denotes its closure.
We regard probability distributions on $[d]$ as row vectors and let
Markov kernels act on the right. Let $e_1,\ldots,e_d$ denote the
standard basis vectors of $\mathbb R^d$, let $\one=(1,\ldots,1)^\mathsf{T}$, and write
\[
u=\frac{1}{d}\one^\mathsf{T},
\qquad
J=\frac{1}{d}\one\one^\mathsf{T}.
\]

For a matrix $A$, we denote its transpose by $A^\mathsf{T}$ and its
conjugate transpose by $A^*$. We write $\tr(A)$ for its trace. For
$A,B\in\mathbb C^{d\times d}$, the Frobenius inner product and norm are
\[
\langle A,B\rangle_F=\tr(A^*B),
\qquad
\|A\|_F^2=\langle A,A\rangle_F.
\]
For vectors in $\mathbb C^d$, $\langle x,y\rangle=x^*y$ denotes the
standard Hermitian inner product.

A Markov kernel on $[d]$ is a matrix
$P\in\mathbb R^{d\times d}$ with nonnegative entries and row sums equal
to one. It is doubly stochastic if its column sums are also equal to
one. A doubly stochastic kernel $P$ is called $r$-regular if
\[
P_{ij}\in\{0,1/r\}
\qquad
\text{for every }i,j\in[d].
\]
Equivalently, every row and every column contains exactly $r$ nonzero
entries, all equal to $1/r$. Such a kernel can be identified with the
transition matrix of the simple random walk on an $r$-regular directed
graph on $[d]$, allowing self-loops.

For $\pi\in S_d$, let $\Pi_\pi \in \mathbb{R}^{d\times d}$ denote the permutation matrix whose action on a row vector $x$ is
\[
(x\Pi_\pi)(j)=x(\pi(j)).
\]
Whenever $\Pi_t$ is a permutation matrix, we write
$\Pi_t=\Pi_{\pi_t}$ for the corresponding permutation $\pi_t\in S_d$.

For a probability measure $\mu$, we write $\Pp_\mu$ and $\E_\mu$ for
probability and expectation with respect to $\mu$. For random variables $X_1,\ldots,X_n$, we write
$\sigma(X_1,\ldots,X_n)$ for the sigma-algebra they generate. We write $\Pois(\lambda)$ for the Poisson distribution with mean
$\lambda$, and $\Bin(n,p)$ for the binomial distribution with parameters
$n$ and $p$.

For probability vectors $\mu,\nu$ on $[d]$, their total-variation
distance is
\[
\|\mu-\nu\|_{\TV}
=
\frac12\sum_{j=1}^d|\mu_j-\nu_j|.
\]
If $\nu_j>0$ for every $j$, their $\chi^2$-divergence is
\[
\chi^2(\mu\|\nu)
=
\sum_{j=1}^d\frac{(\mu_j-\nu_j)^2}{\nu_j}.
\]
For probability measures $\mu,\nu$ on a metric space $(E,d)$
with finite first moments, we define their $1$-Wasserstein distance by
\[
W_1(\mu,\nu)
=
\sup_{\|f\|_{\mathrm{Lip}}\leq1}
\left|
\int_E f\,\dd\mu-\int_E f\,\dd\nu
\right|,
\]
where the supremum is taken over all $1$-Lipschitz functions
$f:E\to\mathbb R$. For probability measures on $\mathbb R$, Vallender \cite{Vallender1974}
gave an equivalent characterization in terms of their distribution
functions. In particular, if $\mu$ and $\nu$ are supported on
$\mathbb N_0=\{0,1,\ldots\}$, then
\begin{equation}\label{eq:wasserstein-discrete}
	W_1(\mu,\nu)
	=
	\sum_{m=0}^\infty
	\left|F_\mu(m)-F_\nu(m)\right|,
\end{equation}
where $F_\mu$ and $F_\nu$ denote the cumulative distribution function of $\mu$ and $\nu$, respectively.

We finish by recording an elementary estimate that will be used
repeatedly.

\begin{lemma}\label{lem:product-telescoping}
	Let $x_1,\ldots,x_r,y_1,\ldots,y_r\in\overline{\mathbb D}$. Then
	\[
	\left|
	\prod_{s=1}^r x_s-\prod_{s=1}^r y_s
	\right|
	\leq
	\sum_{s=1}^r|x_s-y_s|.
	\]
\end{lemma}

\begin{proof}
	Using the telescoping identity
	\[
	\prod_{s=1}^r x_s-\prod_{s=1}^r y_s
	=
	\sum_{t=1}^r
	(x_t-y_t)
	\prod_{s<t}x_s
	\prod_{s>t}y_s,
	\]
	the result follows from the triangle inequality and the fact that all
	factors have modulus at most one.
\end{proof}

\section{Path multiplicities and one-step estimates}\label{sec:one-step}

The technical core of our main results is the probability generating function estimate in Theorem~\ref{thm:complex-pgf}. Its proof proceeds in two stages. First, we quantify the error introduced by a single randomly permuted layer when compared with the idealized with-replacement evolution. We then analyze how these one-step errors accumulate across successive layers. The purpose of this section is to carry out the first stage.

\subsection{Path multiplicities and the empirical generating function}

For $t=1,\ldots,L$, let $P_t \in \mathbb{R}^{d\times d}$ be an $r_t$-regular kernel and let $\Pi_t \in \mathbb{R}^{d\times d}$ be independent uniformly random permutation matrices. Consider the product
\[
R_t=\Pi_1P_1\cdots\Pi_tP_t,
\qquad
R_0=I.
\]
For $i\in[d]$, let $(R_t)_i=e_i^\mathsf{T}R_t$ denote the $i$-th row of $R_t$. We also set
\[
B_t=\prod_{s=1}^t r_s,
\qquad
B_0=1.
\]
We identify an $r$-regular kernel $P$ with the transition matrix of the simple random walk on an $r$-regular directed graph on $[d]$, allowing self-loops. For $i,j\in[d]$, define the incoming and outgoing neighborhoods by
\[
\Gamma_P^-(j)=\{i\in[d]:P_{ij}=1/r\},
\qquad
\Gamma_P^+(i)=\{j\in[d]:P_{ij}=1/r\}.
\]
The $r$-regular hypothesis implies that $|\Gamma_P^-(j)|=|\Gamma_P^+(i)|=r$ for every $i,j\in[d]$. We also let $N_t(i,j)$ denote the number of paths of length $t$ from $i$ to $j$ through the first $t$ layers of the product. Initially,
\[
N_0(i,j)=\ind_{\{i=j\}},
\]
and for $t\geq1$ the multiplicities satisfy
\begin{equation}\label{eq:multiplicity-recursion}
	N_t(i,j)
	=
	\sum_{k\in\Gamma_{P_t}^-(j)}
	N_{t-1}(i,\pi_t(k)).
\end{equation}
Indeed, a path ending at $j$ after the $t$-th layer must use an edge $(k,j)$ of $P_t$ for some $k\in\Gamma_{P_t}^-(j)$. Before applying $P_t$, the permutation $\Pi_t$ places at $k$ the mass that was at $\pi_t(k)$ in the previous layer. Hence the number of paths contributing through $k$ is $N_{t-1}(i,\pi_t(k))$, and summing over the incoming neighbors $k$ gives \eqref{eq:multiplicity-recursion}. See Figure \ref{fig:multiplicity-recursion} for an illustration.

\begin{figure}[b]
	\centering
	\begin{tikzpicture}[
		>=Latex,
		every node/.style={font=\small},
		pile/.style={
			draw,
			rounded corners=2pt,
			minimum width=.85cm,
			minimum height=.62cm,
			fill=white
		},
		perm/.style={
			->,
			dashed,
			gray!65,
			line width=.65pt
		},
		edge/.style={
			->,
			line width=.8pt
		}
		]
		
		
		\node[font=\small\bfseries] at (0,2.55)
		{time $t-1$};
		
		\node[font=\small\bfseries] at (3.7,2.55)
		{after $\pi_t$};
		
		\node[font=\small\bfseries] at (7.2,2.55)
		{time $t$};
		
		
		\node[pile] (A2) at (0,1.45) {$1$};
		\node[pile] (A3) at (0,0.25) {$3$};
		\node[pile] (A1) at (0,-0.95) {$2$};
		
		\node[left=6pt of A2,font=\scriptsize]
		{$\pi_t(k_2)$};
		
		\node[left=6pt of A3,font=\scriptsize]
		{$\pi_t(k_3)$};
		
		\node[left=6pt of A1,font=\scriptsize]
		{$\pi_t(k_1)$};
		
		
		\node[pile] (B1) at (3.7,1.45) {$2$};
		\node[pile] (B2) at (3.7,0.25) {$1$};
		\node[pile] (B3) at (3.7,-0.95) {$3$};
		
		\node[right=6pt of B1,font=\scriptsize]
		{$k_1$};
		
		\node[right=6pt of B2,yshift=5pt,font=\scriptsize]
		{$k_2$};
		
		\node[right=6pt of B3,font=\scriptsize]
		{$k_3$};
		
		\draw[perm] (A1.east) to[out=8,in=188] (B1.west);
		\draw[perm] (A2.east) to[out=-8,in=172] (B2.west);
		\draw[perm] (A3.east) to[out=-4,in=176] (B3.west);
		
		\node[font=\scriptsize,gray!70] at (1.85,2.05)
		{$\pi_t$};
		
		
		\node[pile,line width=.9pt] (J) at (7.2,0.25) {$6$};
		
		\node[right=6pt of J,font=\scriptsize]
		{$j$};
		
		\draw[edge] (B1.east) -- (J.west);
		\draw[edge] (B2.east) -- (J.west);
		\draw[edge] (B3.east) -- (J.west);
		
		\node[font=\scriptsize] at (5.45,2.05)
		{$P_t$};

	\end{tikzpicture}
	
	\caption{One update of the path multiplicities. Here the vertices $k_1$, $k_2$, and $k_3$ are the incoming neighbors of the vertex $j$.}
	\label{fig:multiplicity-recursion}
\end{figure}
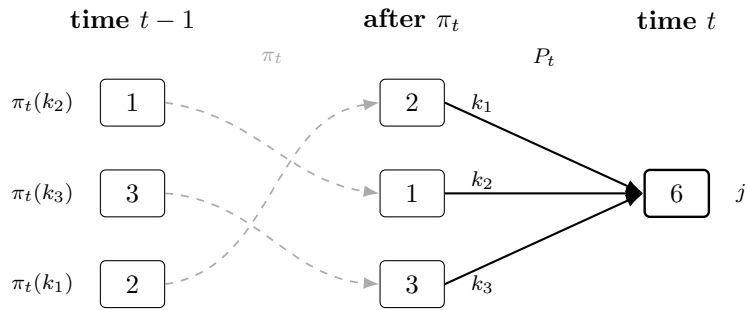

The following elementary identity relates the path multiplicities $N_t(i,j)$ to the entries of the product $R_t$.

\begin{lemma}\label{lem:path-representation}
	For every $t=0,\ldots,L$ and $i,j\in[d]$,
	\begin{equation}\label{eq:path-row-representation}
		(R_t)_{ij}=\frac{N_t(i,j)}{B_t}.
	\end{equation}
\end{lemma}

\begin{proof}
	The statement is immediate for $t=0$. Suppose that \eqref{eq:path-row-representation} holds at time $t-1$. Since $R_t=R_{t-1}\Pi_tP_t$, we have
	\[
		(R_t)_{ij}
		=
		\sum_{k=1}^d (R_{t-1})_{i,\pi_t(k)}(P_t)_{kj}
		=
		\frac{1}{r_tB_{t-1}}
		\sum_{k\in\Gamma_{P_t}^-(j)}
		N_{t-1}(i,\pi_t(k))
		=
		\frac{N_t(i,j)}{B_t},
	\]
	where in the last equality we used \eqref{eq:multiplicity-recursion} and $B_t=r_tB_{t-1}$.
\end{proof}

Since $R_t$ is stochastic, \eqref{eq:path-row-representation} immediately gives
\[
\sum_{j=1}^d N_t(i,j)=B_t
\]
for every $i\in[d]$. Thus each row contains exactly $B_t$ paths, counted with multiplicity. 

To describe how the path multiplicities are distributed among vertices, for each $i\in[d]$ define the random empirical measure
\begin{equation}\label{eq:nu-definition}
	\nu_t^{(i)}
	=
	\frac{1}{d}\sum_{j=1}^d\delta_{N_t(i,j)}.
\end{equation}
Intuitively, $\nu_t^{(i)}$ tells us how common each path multiplicity is among the endpoints. More precisely, for $m\geq0$, $\nu_t^{(i)}(\{m\})$ is the proportion of vertices $j$ with path multiplicity $m$. Since $\sum_j N_t(i,j)=B_t$, its first moment is deterministic and equal to
\[
\int k\,\dd\nu_t^{(i)}(k)=\frac{B_t}{d}.
\]
We analyze $\nu_t^{(i)}$ through its probability generating function
\begin{equation}\label{eq:empirical-pgf}
	M_t^{(i)}(z)
	=
	\int z^k\,\dd\nu_t^{(i)}(k)
	=
	\frac{1}{d}\sum_{j=1}^d z^{N_t(i,j)},
	\qquad |z|\leq1.
\end{equation}
With the purpose of studying $M_t^{(i)}$, we introduce the following operator. Let $P \in \mathbb{R}^{d \times d}$ be an $r$-regular kernel, $x=(x_1,\ldots,x_d)\in \overline{\mathbb D}^d$, and $\pi \in S_d$ be a permutation. Define
\begin{equation}\label{eq:T-definition}
	T_P(x,\pi)
	=
	\frac{1}{d}\sum_{j=1}^d
	\prod_{k\in\Gamma_P^-(j)}x_{\pi(k)}.
\end{equation}
For a fixed $i \in [d]$, this is precisely the object that appears in our dynamics. Indeed, fix $|z|\leq1$ and set
\[
x_j=z^{N_{t-1}(i,j)},\qquad j\in[d].
\]
Then, by \eqref{eq:multiplicity-recursion},
\begin{equation}\label{eq:T-and-M}
T_{P_t}(x,\pi_t) = \frac{1}{d} \sum_{j=1}^d z^{\sum_{k\in\Gamma_{P_t}^-(j)}N_{t-1}(i,\pi_t(k))} = \frac{1}{d} \sum_{j=1}^d z^{N_t(i,j)} = M_t^{(i)}(z)
\end{equation}

\subsection{Analysis of the probability generating function}

We now begin the analysis of the operator $T_P(x,\pi)$. For $x=(x_1,\ldots,x_d)\in\overline{\mathbb D}^d$, write
\[
\overline{x}=\frac{1}{d}\sum_{j=1}^d x_j,
\qquad
q(x)=\frac{1}{d}\sum_{j=1}^d|1-x_j|.
\]
Let $\pi \in S_d$ be a uniformly random permutation. For a fixed $j$, the product
\[
\prod_{k\in\Gamma_P^-(j)}x_{\pi(k)}
\]
is obtained by sampling $r$ coordinates of $x$ uniformly without replacement and multiplying them. If the coordinates were instead sampled independently with replacement, the expectation of the product would simply be $\overline{x}^{\,r}$. 
Our first step is to quantify the error of approximating $\E_\pi T_P(x,\pi)$ by the ideal $\overline{x}^{\,r}$.

\begin{lemma}\label{lem:one-step-bias}
For every $r\leq d/2$,
\begin{equation}\label{eq:one-step-bias}
\left|
\E_\pi T_P(x,\pi)-\overline x^{\,r}
\right|
\leq
\frac{3r^2}{2d}q(x).
\end{equation}
\end{lemma}

\begin{proof}
Let $I_1,\ldots,I_r$ be independent and uniformly distributed on $[d]$. We construct from them a coupled sample without replacement $J_1,\ldots,J_r$ as follows. Set $J_1=I_1$. Having chosen distinct $J_1,\ldots,J_{s-1}$, let
\[
J_s=
\begin{cases}
	I_s, & \text{if } I_s\notin\{J_1,\ldots,J_{s-1}\},\\
	\text{a uniform element of }[d]\setminus\{J_1,\ldots,J_{s-1}\},
	& \text{otherwise}.
\end{cases}
\]
Conditionally on $J_1,\ldots,J_{s-1}$, the variable $J_s$ is uniformly distributed on $[d]\setminus\{J_1,\ldots,J_{s-1}\}$. Hence $(J_1,\ldots,J_r)$ is a uniform sample without replacement from $[d]$. Therefore,
\[
\overline{x}^{\,r}
=
\E\prod_{s=1}^r x_{I_s},
\qquad
\E_\pi T_P(x,\pi)
=
\E\prod_{s=1}^r x_{J_s}.
\]
Since all coordinates have modulus at most one, Lemma \ref{lem:product-telescoping} gives
\[
\left|
\prod_{s=1}^r x_{I_s}-\prod_{s=1}^r x_{J_s}
\right|
\leq
\sum_{s=1}^r|x_{I_s}-x_{J_s}|.
\]
By construction, the $s$-th summand vanishes unless the collision event
\[
C_s:=\{I_s\in\{J_1,\ldots,J_{s-1}\}\}
\]
occurs. Notice that $\mathbb{P}(C_s)=(s-1)/d$. On this event,
\[
|x_{I_s}-x_{J_s}|
\leq
|1-x_{I_s}|+|1-x_{J_s}|.
\]
Since $I_s$ is uniform on $[d]$ and independent of $J_1,\ldots,J_{s-1}$,
\[
\E\left (|1-x_{I_s}|\ind_{C_s}\right )
= \frac{1}{d} \sum_{j=1}^d |1-x_j| \mathbb{E}(\one_{C_s} \mid I_s=j )
= \frac{s-1}{d}q(x).
\]
On the other hand, since $J_s$ is uniform on $[d]\setminus \{J_1,\ldots,J_{s-1}\}$,
\[
\E\left (|1-x_{J_s}|\mid C_s, J_1,\ldots,J_{s-1}\right )
\leq
\frac{\sum_{j=1}^d |1-x_j|}{d-s+1}
=
\frac{dq(x)}{d-s+1}.
\]
Since $\mathbb{P}(C_s\mid J_1,\ldots,J_{s-1})=(s-1)/d$, it follows that
\[
\E\left(|1-x_{J_s}|\ind_{C_s}\right)
\leq
\frac{s-1}{d}\frac{dq(x)}{d-s+1}
\leq
2\frac{s-1}{d}q(x),
\]
where we used $s\leq r\leq d/2$. Therefore,
\[
\E|x_{I_s}-x_{J_s}| = \E(|x_{I_s} - x_{J_s} | \one_{C_s})
\leq
3\frac{s-1}{d}q(x).
\]
Summing over $s=1,\ldots,r$ gives
\[
\left|\E_\pi T_P(x,\pi)-\overline{x}^{\,r}\right|
\leq
\frac{3r(r-1)}{2d}q(x)
\leq
\frac{3r^2}{2d}q(x).\qedhere
\]
\end{proof}

We next control the fluctuations using the following standard Poincar\'e inequality for the random transposition walk on the symmetric group.

\begin{lemma}[Poincar\'e inequality on $S_d$]
	\label{lem:permutation-poincare}
	Let $\pi$ be uniform on $S_d$, and let $(k,\ell)$ be uniform over the ordered pairs of distinct elements of $[d]$, independently of $\pi$. If $\tau_{k,\ell}$ denotes the transposition of $k$ and $\ell$, then, for every $f:S_d\to\mathbb C$,
	\[
	\Var_\pi(f(\pi))
	\leq
	\frac{d-1}{4}
	\E_{\pi,(k,\ell)}
	\left|f(\pi)-f(\pi\circ\tau_{k,\ell})\right|^2.
	\]
\end{lemma}

Indeed, in the simple random transposition walk each transposition is chosen with probability $2/(d(d-1))$ and its spectral gap is $2/(d-1)$; see
\cite[Section~1.3]{ErbarMaasTetali2015}.
The random transposition walk and its spectral analysis go back to
Diaconis and Shahshahani \cite{DiaconisShahshahani1981}, who use a
closely related lazy normalization.

Taking $f=T_P(x,\cdot)$ in Lemma~\ref{lem:permutation-poincare}, we can therefore estimate its variance by controlling how much it changes when two coordinates are transposed. This gives the following estimate.

\begin{lemma}\label{lem:one-step-variance}
	Let $P$ be an $r$-regular kernel. For every $x\in\overline{\mathbb D}^d$,
	\begin{equation}\label{eq:one-step-variance}
		\Var_\pi\bigl(T_P(x,\pi)\bigr)
		\leq
		\frac{4r^3}{d}q(x)^2.
	\end{equation}
\end{lemma}

\begin{proof}
	Let $\tau_{k,\ell}$ be the transposition of indices $k, \ell$. By Lemma \ref{lem:permutation-poincare}, it is enough to control the impact in $T_P(x,\pi)$ of swapping indices $k$ and $\ell$. Write
	\[
	X=x_{\pi(k)},
	\qquad
	Y=x_{\pi(\ell)}.
	\]
	Recall that
	\[ T_P(x,\pi) = \frac{1}{d} \sum_{j=1}^d \prod_{v\in \Gamma_P^{-}(j)} x_{\pi(v)}.\]
	Swapping $k$ and $\ell$ affects only the terms corresponding to the indices $j$ for which either $k\in \Gamma_P^{-}(j)$ or $\ell\in \Gamma_P^{-}(j)$. Such indices $j$ are precisely the outgoing neighbors of $k$ and $\ell$. Furthermore, note that if $j\in\Gamma_P^+(k)\cap\Gamma_P^+(\ell)$, then both $k$ and $\ell$ appear in the corresponding product, and swapping them leaves the product unchanged by commutativity. Hence only the indices belonging to exactly one of the two outgoing neighborhoods contribute to the difference. Set
	\[
	A=\Gamma_P^+(k)\setminus\Gamma_P^+(\ell),
	\qquad
	B=\Gamma_P^+(\ell)\setminus\Gamma_P^+(k).
	\]
	For $j\in A$ define the product
	\[
	U_j=\prod_{v\in\Gamma_P^-(j)\setminus\{k\}}x_{\pi(v)}.
	\]
	Similarly, for $j\in B$ define
	\[
	V_j=\prod_{v\in\Gamma_P^-(j)\setminus\{\ell\}}x_{\pi(v)}.
	\]
	A direct computation gives
	\begin{align*}
		T_P(x,\pi)-T_P(x,\pi\circ\tau_{k,\ell})
		&=
		\frac{1}{d}\sum_{j=1}^d
		\left(
		\prod_{v\in\Gamma_P^-(j)}x_{\pi(v)}
		-
		\prod_{v\in\Gamma_P^-(j)}x_{\pi\circ\tau_{k,\ell}(v)}
		\right)\\
		&=
		\frac{1}{d}\sum_{j\in A}(x_{\pi(k)}-x_{\pi(\ell)})U_j
		+
		\frac{1}{d}\sum_{j\in B}(x_{\pi(\ell)}-x_{\pi(k)})V_j\\
		&=
		\frac{X-Y}{d}
		\left(
		\sum_{j\in A}U_j-\sum_{j\in B}V_j
		\right).
	\end{align*}
	The regularity assumption implies that
	$|A|=|B|$ and thus there is a bijection $\phi \colon A\longrightarrow B$. By Cauchy--Schwarz inequality,
	\begin{equation}\label{eq:variance-c-s}
	\left|
	\sum_{j\in A}U_j-\sum_{j\in B}V_j
	\right|^2 = \left | \sum_{j \in A} \bigl (U_j - V_{\phi(j)} \bigr ) \right |^2
	\leq
	|A|\sum_{j\in A}|U_j-V_{\phi(j)}|^2.
	\end{equation}
	Set $a_v=|1-x_v|$. Since $|U_j|,|V_{\phi(j)}|\leq1$, we have
	\[
	|U_j-V_{\phi(j)}|^2
	\leq 2|U_j - V_{\phi(j)}|
	\leq
	2\bigl(|1-U_j|+|1-V_{\phi(j)}|\bigr).
	\]
	Moreover, for $w_1,\ldots,w_s\in\overline{\mathbb D}$, Lemma \ref{lem:product-telescoping} gives
	\[
	\left|1-\prod_{v=1}^s w_v\right|
	\leq
	\sum_{v=1}^s|1-w_v|.
	\]
	Consequently, by \eqref{eq:variance-c-s} and $|A|\leq r$,
	\[
	\left|
	\sum_{j\in A}U_j-\sum_{j\in B}V_j
	\right|^2
	\leq
	2|A| D_\pi\leq 2rD_\pi,
	\]
	where
	\[
	D_\pi
	=
	\sum_{j\in A}
	\sum_{v\in\Gamma_P^-(j)\setminus\{k\}}a_{\pi(v)}
	+
	\sum_{j\in B}
	\sum_{v\in\Gamma_P^-(j)\setminus\{\ell\}}a_{\pi(v)}.
	\]
	Similarly,
	\[
	|X-Y|^2
	\leq 2|X-Y| \leq
	2\bigl(a_{\pi(k)}+a_{\pi(\ell)}\bigr).
	\]
	Thus, 
	\begin{equation}\label{eq:variance-final}
	\left|T_P(x,\pi)-T_P(x,\pi\circ\tau_{k,\ell})\right|^2
	\leq
	\frac{4r}{d^2}
	\bigl(a_{\pi(k)}+a_{\pi(\ell)}\bigr)D_\pi.
	\end{equation}
	Each product $U_j$ and $V_j$ contains exactly $r-1$ factors. Hence $D_\pi$ is a sum of
	\[
	2|A|(r-1)\leq 2r^2
	\]
	terms of the form $a_{\pi(v)}$, where in every case $v\notin\{k,\ell\}$.
	
	For any two distinct indices $s,v\in[d]$, the pair
	$(\pi(s),\pi(v))$ is uniformly distributed over all ordered pairs of distinct elements of $[d]$. Therefore,
	\[
	\E_\pi\bigl[a_{\pi(s)}a_{\pi(v)}\bigr]
	=
	\frac{1}{d(d-1)}
	\sum_{i\neq j}a_ia_j
	\leq
	\frac{1}{d(d-1)}
	\left(\sum_{j=1}^d a_j\right)^2
	=
	\frac{d}{d-1}q(x)^2.
	\]
	Since every index $v$ appearing in $D_\pi$ is distinct from both $k$ and $\ell$, it follows that
	\[
	\E_\pi\left[
	\bigl(a_{\pi(k)}+a_{\pi(\ell)}\bigr)D_\pi
	\right]
	\leq
	4r^2\frac{d}{d-1}q(x)^2.
	\]
	Consequently, taking expectation in \eqref{eq:variance-final} yields
	\[
	\E_\pi
	\left|T_P(x,\pi)-T_P(x,\pi\circ\tau_{k,\ell})\right|^2
	\leq
	\frac{16r^3}{d(d-1)}q(x)^2.
	\]
	The result now follows from Lemma \ref{lem:permutation-poincare}.
\end{proof}

\section{Proof of Theorem~\ref{thm:complex-pgf}}\label{sec:stability}

This section is dedicated to proving Theorem~\ref{thm:complex-pgf}. Fix $i\in[d]$. Throughout this section, to simplify notation, we write
\[
N_t(j)=N_t(i,j),
\qquad
\nu_t=\nu_t^{(i)},
\qquad
M_t(z)=M_t^{(i)}(z).
\]
We also let
\[
\mathcal F_t=\sigma(\pi_1,\ldots,\pi_t),
\qquad
B_t=\prod_{s=1}^t r_s,
\qquad
\lambda_t=\frac{B_t}{d}.
\]
Recall that
\[
M_t(z)=\frac{1}{d}\sum_{j=1}^d z^{N_t(j)}
\]
is the empirical probability generating function of the path multiplicities after $t$ layers. In the previous section we showed that, conditionally on $\mathcal F_{t-1}$, its one-step evolution is given by
\[
M_t(z)=T_{P_t}(x,\pi_t),
\]
where $T_{P_t}$ is the operator defined in (\ref{eq:T-definition}) and $x_j=z^{N_{t-1}(j)}$. Moreover,
\[
\overline{x}=\frac{1}{d} \sum_{j=1}^d x_j=M_{t-1}(z).
\]
If the coordinates in each product defining $T_{P_t}$ were sampled independently with replacement, the corresponding conditional expectation would therefore be
\[
\overline{x}^{\,r_t}
=
M_{t-1}(z)^{r_t}.
\]
This suggests comparing $M_t(z)$ with
\begin{equation}\label{eq:binomial-comparator}
	m_t(z)
	=
	\left(1+\frac{z-1}{d}\right)^{B_t}.
\end{equation}
Indeed,
\[
M_0(z)=1+\frac{z-1}{d}=m_0(z),
\qquad
m_t(z)=m_{t-1}(z)^{r_t}.
\]
Thus $m_t$ is precisely the generating function obtained by iterating the idealized with-replacement recursion. Equivalently, $m_t$ is the probability generating function of a $\operatorname{Bin}(B_t,1/d)$ random variable.

Lemmas~\ref{lem:one-step-bias} and~\ref{lem:one-step-variance} will be used to control the discrepancy between the actual update $M_t$ and the idealized update $M_{t-1}^{r_t}$. The proof of Theorem~\ref{thm:complex-pgf} then consists of understanding how these errors accumulate through the iteration. To this end, define
\begin{equation}\label{eq:Delta-definition}
	\Delta_t(z)
	=
	\frac{M_t(z)-m_t(z)}{B_t}.
\end{equation}
Before turning to the proof of the theorem, we establish a simple deterministic bound for $\Delta_t$.

\begin{lemma}\label{lem:deterministic-delta}
	For every $t=0,\ldots,L$ and $z\in\overline{\mathbb D}$,
	\begin{equation}\label{eq:deterministic-delta}
		|\Delta_t(z)|
		\leq
		\frac{2|1-z|}{d}.
	\end{equation}
\end{lemma}

\begin{proof}
	The result will follow from the elementary inequality $|1-z^n|\leq n|1-z|$ for $|z|\leq1$, which is an application of Lemma \ref{lem:product-telescoping}. First, since $\sum_{j=1}^d N_t(j)=B_t$,
	\[
	|1-M_t(z)|
	\leq
	\frac{1}{d}\sum_{j=1}^d |1-z^{N_t(j)}|
	\leq
	|1-z|\frac{B_t}{d}.
	\]
	On the other hand, the same inequality gives
	\[
	|1-m_t(z)| = \left | 1- \left (1+\frac{z-1}{d}\right )^{B_t} \right |
	\leq
	B_t\left|\frac{z-1}{d}\right|
	=
	|1-z|\frac{B_t}{d}.
	\]
	Consequently,
	\[
	|M_t(z)-m_t(z)|
	\leq
	\frac{2|1-z|B_t}{d}
	\]
	and the result follows.
\end{proof}

The final ingredient is a standard comparison between $m_t$ and the probability generating function of a Poisson random variable with the same mean. For $\lambda\geq0$, let
\[
\Phi_\lambda(z)=e^{\lambda(z-1)}
\]
denote the probability generating function of a $\operatorname{Pois}(\lambda)$ random variable, defined for $z \in \overline{\mathbb{D}}$.

\begin{lemma}\label{lem:binomial-poisson-pgf}
	For every $d\geq1$, $t=0,\ldots,L$, and $z\in\overline{\mathbb D}$,
	\begin{equation}\label{eq:binomial-poisson-pgf}
		\left|
		m_t(z)-\Phi_{\lambda_t}(z)
		\right|
		\leq
		\frac{\lambda_t}{2d}|1-z|^2.
	\end{equation}
\end{lemma}

\begin{proof}
	Set $a=1-z$. Since $|z|\leq1$,
	\[
	\left|1-\frac{a}{d}\right|
	=
	\left|
	\left(1-\frac{1}{d}\right)+\frac{z}{d}
	\right|
	\leq1,
	\qquad
	\left|e^{-a/d}\right|\leq1,
	\]
	where the second inequality uses that $a$ has nonnegative real part. Using Lemma~\ref{lem:product-telescoping} and
	$\lambda_t=B_t/d$, we obtain
	\[
		\left|m_t(z)-\Phi_{\lambda_t}(z)\right|
		=
		\left|
		\left(1-\frac{a}{d}\right)^{B_t}
		-
		\left(e^{-a/d}\right)^{B_t}
		\right|
		\leq
		B_t
		\left|
		1-\frac{a}{d}-e^{-a/d}
		\right|.
	\]
	It remains to bound the Taylor remainder of the exponential. Consider
	\[
	g(s)=e^{-sa/d},
	\qquad 0\leq s\leq1.
	\]
	By the Taylor series remainder formula,
	\[
	|g(1)-g(0)-g'(0)|
	\leq
	\frac12\sup_{0\leq s\leq1}|g''(s)|.
	\]
	For any $0\leq s \leq 1$,
	\[
	|g''(s)|=\frac{|a|^2}{d^2}|e^{-sa/d}| \leq \frac{|a|^2}{d^2}.
	\]
	Therefore,
	\[
	\left|
	1-\frac{a}{d}-e^{-a/d}
	\right|
	\leq
	\frac{|a|^2}{2d^2}.
	\]
	Combining the preceding estimates gives
	\[
	\left|m_t(z)-\Phi_{\lambda_t}(z)\right|
	\leq
	\frac{B_t|a|^2}{2d^2}
	=
	\frac{\lambda_t}{2d}|1-z|^2.\qedhere
	\]
\end{proof}

We are now ready to prove Theorem~\ref{thm:complex-pgf}. 

\begin{proof}[Proof of Theorem~\ref{thm:complex-pgf}]
	Fix $z\in\overline{\mathbb D}$. At the $t$-th layer, we compare $M_t$ with $M_{t-1}^{r_t}$. We do this by inserting the conditional expectation as an intermediate quantity. Define
	\[
	\beta_t
	=
	\E\left[M_t\mid\mathcal F_{t-1}\right]-M_{t-1}^{r_t},
	\qquad
	\varepsilon_t
	=
	M_t-\E\left[M_t\mid\mathcal F_{t-1}\right].
	\]
	Thus
	\begin{equation}\label{eq:one-step-decomposition}
		M_t-M_{t-1}^{r_t}
		=
		\beta_t+\varepsilon_t,
	\end{equation}
	where $\beta_t$ measures the discrepancy between the conditional mean of the actual update and the idealized update, while $\varepsilon_t$ is the centered fluctuation around this conditional mean. In particular, $\E[\varepsilon_t\mid\mathcal F_{t-1}]=0$. Recall that $m_t=m_{t-1}^{r_t}$. From \eqref{eq:one-step-decomposition},
	\[
	M_{t}-m_t = M_{t-1}^{r_t} - m_{t-1}^{r_t} + \beta_t + \varepsilon_t.
	\]
	Dividing by $B_t=r_t B_{t-1}$ gives the recursion
	\begin{equation}\label{eq:Delta-recursion}
		\Delta_t = \gamma_t \Delta_{t-1} + b_t + \zeta_t,
	\end{equation}
	where
	\[
	\gamma_t=\frac{M_{t-1}^{r_t} - m_{t-1}^{r_t}}{r_t(M_{t-1} - m_{t-1})},
	\qquad
	b_t=\frac{\beta_t}{B_t},
	\qquad
	\zeta_t=\frac{\varepsilon_t}{B_t}.
	\]
	The apparent singularity in the definition of $\gamma_t$ when $M_{t-1}=m_{t-1}$ is removable, so $\gamma_t$ is defined there by continuous extension. Next, we square and expand the error. Since $\gamma_t$, $\Delta_{t-1}$, and $b_t$ are
	$\mathcal F_{t-1}$-measurable and
	$\E[\zeta_t\mid\mathcal F_{t-1}]=0$, taking conditional second moments in (\ref{eq:Delta-recursion}) gives
	\begin{equation}\label{eq:delta-second-moment}
		\E\left[|\Delta_t|^2\mid\mathcal F_{t-1}\right]
		=
		|\gamma_t\Delta_{t-1}+b_t|^2
		+
		\E\left[|\zeta_t|^2\mid\mathcal F_{t-1}\right].
	\end{equation}
	
	It remains to estimate the size of quantities appearing in \eqref{eq:delta-second-moment}. First, since $|M_{t-1}(z)|\leq1$ and $|m_{t-1}(z)|\leq1$, from Lemma \ref{lem:product-telescoping} we obtain
	\[|\gamma_t|\leq1.\]
	Next, since $S_L\leq1/2$, we have $r_t\leq d/2$. Set
	$x\in\overline{\mathbb D}^d$ by $x_j=z^{N_{t-1}(j)}$. Then
	\[
	T_{P_t}(x,\pi_t)=M_t(z),
	\qquad
	\E_{\pi_t}T_{P_t}(x,\pi_t)
	=
	\E\left[M_t(z)\mid\mathcal F_{t-1}\right],
	\qquad
	\overline{x}=M_{t-1}(z).
	\]
	Hence, \cref{lem:one-step-bias} gives
	\[
	|b_t|
	=
	\frac{|\beta_t|}{B_t}
	\leq
	\frac{3r_t^2}{2d}\frac{q(x)}{B_t}.
	\]
	Similarly, \cref{lem:one-step-variance} yields
	\[
	\E\left[|\zeta_t|^2\mid\mathcal F_{t-1}\right]
	=
	\frac{1}{B_t^2}
	\Var_{\pi_t}\bigl(T_{P_t}(x,\pi_t)\bigr)
	\leq
	\frac{4r_t^3}{d}\frac{q(x)^2}{B_t^2}.
	\]
	Finally, using $|1-z^n|\leq n|1-z|$ and
	$\sum_{j=1}^dN_{t-1}(j)=B_{t-1}$,
	\begin{equation}\label{eq:q-dynamics}
		q(x)
		=
		\frac{1}{d}\sum_{j=1}^d|1-z^{N_{t-1}(j)}|
		\leq
		|1-z|\frac{B_{t-1}}{d} = |1-z|\frac{B_t}{dr_t}.
	\end{equation}
	Substituting \eqref{eq:q-dynamics} into the bounds for $b_t$ and $\zeta_t$ gives
	\[
	|b_t|
	\leq
	\frac{3|1-z|r_t}{2d^2},
	\qquad
	\E\left(|\zeta_t|^2\mid\mathcal F_{t-1}\right)
	\leq
	\frac{4|1-z|^2r_t}{d^3}.
	\]
	Using these estimates, together with $|\gamma_t|\leq1$ and
	\cref{lem:deterministic-delta}, from \eqref{eq:delta-second-moment} we obtain
	\[
	\begin{aligned}
		\E\left[|\Delta_t|^2\mid\mathcal F_{t-1}\right]
		&\leq
		|\Delta_{t-1}|^2
		+
		2|\Delta_{t-1}||b_t|
		+
		|b_t|^2
		+
		\E\left[|\zeta_t|^2\mid\mathcal F_{t-1}\right]\\
		&\leq
		|\Delta_{t-1}|^2
		+
		\frac{6|1-z|^2r_t}{d^3}
		+
		\frac{9|1-z|^2r_t^2}{4d^4}
		+
		\frac{4|1-z|^2r_t}{d^3}.
	\end{aligned}
	\]
	Since $r_t\leq d/2$,
	\[
	\frac{9r_t^2}{4d^4}
	\leq
	\frac{2r_t}{d^3},
	\]
	and hence
	\begin{equation}\label{eq:delta-recursion-bound}
		\E\left[|\Delta_t|^2\mid\mathcal F_{t-1}\right]
		\leq
		|\Delta_{t-1}|^2
		+
		12\frac{|1-z|^2r_t}{d^3}.
	\end{equation}
	Taking expectations and iterating \eqref{eq:delta-recursion-bound}, using $\Delta_0=0$, gives
	\[
	\E|\Delta_L|^2
	\leq
	12\frac{|1-z|^2}{d^3}
	\sum_{t=1}^L r_t
	=
	12\frac{|1-z|^2S_L}{d^2}.
	\]
	Finally, since $M_L-m_L=B_L\Delta_L$ and $B_L=\lambda_Ld$,
	\begin{equation}\label{eq:binomial-stability}
		\E|M_L(z)-m_L(z)|^2
		\leq
		12\lambda_L^2|1-z|^2S_L.
	\end{equation}
	
	It remains to compare $m_L$ to the probability generating function of the Poisson distribution. By Lemma~\ref{lem:binomial-poisson-pgf},
	\[
	|m_L(z)-e^{\lambda_L(z-1)}|
	\leq
	\frac{\lambda_L}{2d}|1-z|^2.
	\]
	Hence, by the triangle inequality in $L^2$,
	\[
	\begin{aligned}
		\left(
		\E\left|
		M_L(z)-e^{\lambda_L(z-1)}
		\right|^2
		\right)^{1/2}
		&\leq
		\sqrt{12}
		\lambda_L |1-z|\sqrt{S_L}
		+
		\frac{\lambda_L|1-z|^2}{2d}.
	\end{aligned}
	\]
	Since $|1-z|\leq2$  and $1/d\leq S_L\leq1/2$, we have
	\[
	\frac{|1-z|}{2d}\leq S_L \leq \sqrt{S_L/2}.
	\]
	Therefore,
	\[
	\left(
	\E\left|
	M_L(z)-e^{\lambda_L(z-1)}
	\right|^2
	\right)^{1/2}
	\leq
	\left(\sqrt{12}+\frac{1}{\sqrt{2}}\right)
	\lambda_L |1-z|\sqrt{S_L}\leq \sqrt{18} \lambda_L |1-z|\sqrt{S_L}.
	\]
	Squaring both sides proves the theorem.
\end{proof}

\section{Quantitative Poisson approximation}\label{sec:poisson-metrics}

In this section we derive \cref{cor:local-poisson,cor:poisson-berry-esseen,cor:wasserstein-poisson} from Theorem~\ref{thm:complex-pgf}. The main idea is to use Fourier inversion on the unit circle. The probabilities associated with the empirical multiplicity measure are the Fourier coefficients of its probability generating function, while cumulative probabilities can be recovered via Abel summation. The estimate of Theorem~\ref{thm:complex-pgf} then provides quantitative control of these coefficients.

We begin with Theorem~\ref{cor:local-poisson}. Fix $i\in[d]$ and write $\lambda=\lambda_L$. To simplify notation, let
\[
\nu=\nu_L^{(i)},
\qquad
M(z)=M_L^{(i)}(z),
\]
and let $Z\sim\operatorname{Pois}(\lambda)$. For $k\geq0$, define
\[
p_k=\Pp(Z=k),
\qquad
\widehat{\nu}_k=\nu(\{k\}),
\qquad
a_k=\widehat{\nu}_k-p_k.
\]
Define $H\colon  \overline{\mathbb{D}}\longrightarrow \mathbb{C}$ by $H(z)=M(z)-e^{\lambda(z-1)}$. Then
\begin{equation}\label{eq:H-definition}
	H(z)
	=
	\sum_{k=0}^\infty a_kz^k.
\end{equation}
Thus the errors $a_k$ in the point masses are precisely the Fourier coefficients of $H$ on the unit circle. This gives an integral representation for $a_k$, which can be estimated using Theorem~\ref{thm:complex-pgf}.

\begin{proof}[Proof of \cref{cor:local-poisson}]
	By the Fourier coefficient formula,
	\[
	a_k
	=
	\frac{1}{2\pi}
	\int_{-\pi}^{\pi}
	H(e^{it})e^{-ikt}\,\dd t.
	\]
	Hence
	\[
	\sup_{k\geq0}|a_k|
	\leq
	\frac{1}{2\pi}
	\int_{-\pi}^{\pi}|H(e^{it})|\,\dd t.
	\]
	Taking expectations and using Cauchy--Schwarz inequality and Theorem \ref{thm:complex-pgf},
	\[
		\E\sup_{k\geq0}|a_k|
		\leq
		\frac{1}{2\pi}
		\int_{-\pi}^{\pi}
		\bigl(\E|H(e^{it})|^2\bigr)^{1/2}\,\dd t
		\leq
		\sqrt{18}\lambda_L\sqrt{S_L}
		\frac{1}{2\pi}
		\int_{-\pi}^{\pi}|1-e^{it}|\,\dd t
		=
		\frac{12\sqrt{2}}{\pi}\lambda_L\sqrt{S_L}.\qedhere
	\]
\end{proof}

We next prove Theorem~\ref{cor:poisson-berry-esseen}. To pass from point probabilities to cumulative distribution functions, define
\begin{equation}\label{eq:c_m}
c_m
=
\sum_{k=0}^m a_k
=
\nu([0,m])-\Pp(Z\leq m),
\qquad m\geq0.
\end{equation}
Since $\nu$ and $\Pois(\lambda)$ have the same mean $\lambda$,
\[
\sum_{m=0}^\infty |c_m|
\leq
\sum_{m=0}^\infty \nu((m,\infty))
+
\sum_{m=0}^\infty \Pp(Z>m)
=
2\lambda.
\]
Thus the series $\sum_{m\geq0}c_mz^m$ converges absolutely on
$\overline{\mathbb D}$. Furthermore, since $a_k=c_k-c_{k-1}$, with the convention $c_{-1}=0$, for $|z|\leq1$ we have
\[
\begin{aligned}
	H(z)
	&=
	\sum_{k=0}^\infty a_k z^k
	=
	\sum_{k=0}^\infty (c_k-c_{k-1})z^k\\
	&=
	\sum_{k=0}^\infty c_kz^k
	-
	z\sum_{k=0}^\infty c_kz^k
	=
	(1-z)\sum_{k=0}^\infty c_kz^k.
\end{aligned}
\]
Thus, for any $|z|\leq 1$ with $z\neq 1$ we have
\begin{equation}\label{eq:abel-cdf}
	\frac{H(z)}{1-z}
	=
	\sum_{k=0}^\infty c_kz^k.
\end{equation}
The right-hand side is continuous on $\overline{\mathbb D}$, so the quotient extends continuously to $z=1$. In fact, since the two measures have the same mean, $\sum_{k}c_k=0$.

In other words, the errors $c_m$ in the cumulative distribution functions are precisely the Fourier coefficients of $H(z)/(1-z)$.  Moreover, the factor $|1-z|$ in Theorem~\ref{thm:complex-pgf} cancels the denominator in \eqref{eq:abel-cdf}. Therefore, the same Fourier argument used in the proof of Theorem~\ref{cor:local-poisson} yields a uniform bound for the cumulative distribution functions.

\begin{proof}[Proof of \cref{cor:poisson-berry-esseen}]
	By \eqref{eq:abel-cdf}, the coefficients $c_m$ are the Fourier coefficients of $H(z)/(1-z)$. Hence, for every $m\geq0$,
	\[
	c_m
	=
	\frac{1}{2\pi}
	\int_{-\pi}^{\pi}
	\frac{H(e^{it})}{1-e^{it}}
	e^{-imt}\,\dd t.
	\]
	Therefore,
	\[
	\sup_{m\geq0}|c_m|
	\leq
	\frac{1}{2\pi}
	\int_{-\pi}^{\pi}
	\frac{|H(e^{it})|}{|1-e^{it}|}\,\dd t.
	\]
	Taking expectations and using Cauchy--Schwarz inequality and Theorem \ref{thm:complex-pgf},
	\[
		\E\sup_{m\geq0}|c_m|
		\leq
		\frac{1}{2\pi}
		\int_{-\pi}^{\pi}
		\frac{\bigl(\E|H(e^{it})|^2\bigr)^{1/2}}
		{|1-e^{it}|}\,\dd t
		\leq
		\frac{1}{2\pi}
		\int_{-\pi}^{\pi}
		\sqrt{18}\lambda_L\sqrt{S_L}\,\dd t
		=
		\sqrt{18}\lambda_L\sqrt{S_L}.\qedhere
	\]
\end{proof}

Finally, we turn to Theorem~\ref{cor:wasserstein-poisson}. By the characterization
\eqref{eq:wasserstein-discrete}, we have
\[
W_1\bigl(\nu,\Pois(\lambda)\bigr)
=
\sum_{m=0}^\infty |c_m|,
\]
where $c_m$ are defined as in (\ref{eq:c_m}). By \eqref{eq:abel-cdf}, $c_m$ are precisely the Fourier coefficients of $H(z)/(1-z)$. Parseval's identity then allows us to control their $\ell_2$-norm directly from Theorem~\ref{thm:complex-pgf}. To pass from $\ell_2$ to $\ell_1$, we truncate the sequence at a level $N$: the initial segment is bounded by Cauchy--Schwarz inequality, while the tail is controlled using the equality of the first moments and the Poisson tail. Optimizing over $N$ yields the Wasserstein estimate, at the cost of the logarithmic factor appearing in \eqref{eq:wasserstein-bound-main}.

\begin{proof}[Proof of \cref{cor:wasserstein-poisson}]
	Let $c_m$ be defined as in \eqref{eq:c_m}. Then
	\eqref{eq:wasserstein-discrete} gives
	\[
	W_1\bigl(\nu,\Pois(\lambda)\bigr)
	=
	\sum_{m=0}^\infty |c_m|.
	\]
	We first control the $\ell_2$-norm of $(c_m)_{m\geq0}$. By
	\eqref{eq:abel-cdf}, the coefficients $c_m$ are the Fourier coefficients
	of $H(z)/(1-z)$. Hence, by Parseval's identity and
	Theorem~\ref{thm:complex-pgf}, 
	\[
		\E\sum_{m=0}^\infty |c_m|^2
		=
		\frac{1}{2\pi}
		\int_{-\pi}^{\pi}
		\E\left|
		\frac{H(e^{it})}{1-e^{it}}
		\right|^2\,\dd t
		\leq
		18\lambda_L^2 S_L.
	\]
	For every integer $N\geq1$, Cauchy--Schwarz inequality gives
	\begin{equation}\label{eq:wasserstein-head}
		\E\sum_{m=0}^{N-1}|c_m|
		\leq
		\sqrt{N}
		\left(
		\E\sum_{m=0}^\infty |c_m|^2
		\right)^{1/2}
		\leq
		\sqrt{18}\lambda_L\sqrt{NS_L}.
	\end{equation}
	Next, we control the tail. Let
	\[
	T_\nu(m)=\nu((m,\infty)),
	\qquad
	T_Z(m)=\Pp(Z>m).
	\]
	Since $\nu$ and $\operatorname{Pois}(\lambda)$ have the same first moment $\lambda$, we have
	\[
	\sum_{m=0}^\infty T_\nu(m)
	=
	\sum_{m=0}^\infty T_Z(m).
	\]
	Moreover, $c_m=T_Z(m)-T_\nu(m)$. Consequently,
	\[
	\sum_{m=N}^\infty T_\nu(m)
	=
	\sum_{m=N}^\infty T_Z(m)
	+
	\sum_{m=0}^{N-1}c_m.
	\]
	Using $|c_m|\leq T_\nu(m)+T_Z(m)$, it follows that
	\[
	\sum_{m=N}^\infty|c_m|
	\leq
	2\sum_{m=N}^\infty T_Z(m)
	+
	\sum_{m=0}^{N-1}|c_m|.
	\]
	Combining this with \eqref{eq:wasserstein-head},
	\begin{equation}\label{eq:wasserstein-truncation}
		\E W_1\bigl(\nu,\operatorname{Pois}(\lambda)\bigr)
		\leq
		2\sqrt{18}\lambda_L\sqrt{NS_L}
		+
		2\sum_{m=N}^\infty\Pp(Z>m).
	\end{equation}
For $Z\sim\Pois(\lambda)$ and every integer $N>\lambda$,
\[
\sum_{m=N}^\infty \Pp(Z>m)
=
\E(Z-N)_+
\leq
\E(Z \one_{Z>N})
=
\lambda\Pp(Z\geq N)
\leq
\lambda e^{-\lambda}
\left(\frac{e\lambda}{N}\right)^N.
\]
Assume $S_L\leq1/4$ and choose
\[
N
=
\left\lceil
(1+\lambda)
\frac{\log(1/S_L)}
{\log\log(1/S_L)}
\right\rceil.
\]
Then $N>\lambda$, and substituting $N$ into the preceding Poisson tail estimate gives
\[
\sum_{m=N}^\infty \Pp(Z>m)
\leq
\lambda e^{-\lambda}\left(\frac{e\lambda}{N}\right)^N
\leq \lambda S_L.
\]
To see this, set $a=\log(1/S_L)>1$. By the definition of $N$, we have $N\log a\geq(1+\lambda)a$. Hence, by the Poisson Chernoff bound,
\[
e^{-\lambda}\left(\frac{e\lambda}{N}\right)^N
\leq
e^{\lambda(a-1)-N\log a}
\leq e^{-a-\lambda}
\leq S_L.
\]
Using this estimate in (\ref{eq:wasserstein-truncation}) yields
\[
\E W_1\bigl(\nu,\Pois(\lambda)\bigr)
\leq 2\sqrt{18}\lambda_L\sqrt{NS_L}
+
2\lambda S_L \leq
14\lambda\sqrt{1+\lambda}
\sqrt{
	S_L\frac{\log(1/S_L)}
	{\log\log(1/S_L)}
},
\]
which proves the quantitative estimate.
	
Finally, suppose that $\lambda_L\to\lambda$ and $S_L\to0$. By the additivity of Poisson random variables,
\[
	W_1\bigl(\Pois(\lambda_L),\Pois(\lambda)\bigr)
	=
	|\lambda_L-\lambda|.
\]
Hence, by the triangle inequality,
\[
	\E W_1\bigl(\nu_L^{(i)},\Pois(\lambda)\bigr)
	\longrightarrow0.
\]
Thus the convergence holds in $L^1$.
\end{proof}

\section{Permutation isotropy and exact \texorpdfstring{$L^2$}{L2} identities}\label{sec:l2}

The main goal of this section is to prove Theorem~\ref{thm:exact-covariance} and Corollary~\ref{cor:exact-l2}. We then apply these identities to the regular kernel model analyzed in earlier sections, which provides a concrete interpretation of the contraction factors and identifies the relevant mixing scale.

The first step is to analyze the effect of conjugating a matrix $M \in \mathbb{C}^{d\times d}$ by a uniform random permutation. We will see that this produces an isotropic averaging effect: all diagonal entries are averaged together, and likewise all off-diagonal entries are averaged together.

\begin{lemma}\label{lem:permutation-isotropy}
	Let $M\in\C^{d\times d}$ satisfy $M\one=0$, and let $\Pi$ be a uniform permutation matrix. Then
	\begin{equation}\label{eq:isotropy}
		\E\left(\Pi M\Pi^\mathsf{T}\right)
		=
		\frac{\tr M}{d-1}(I-J).
	\end{equation}
\end{lemma}

\begin{proof}
	The matrix $\overline M=\E(\Pi M\Pi^\mathsf{T})$ is invariant under conjugation by every permutation matrix. Indeed, if $S$ is a permutation matrix, then $S\Pi$ is also a uniform permutation matrix and thus
	\[ S \overline M S^\T = \mathbb{E} \left ( S\Pi M \Pi^\T S^\T\right ) = \mathbb{E}\left(\Pi M\Pi^\mathsf{T}\right) = \overline{M}.\]
	Hence all its diagonal entries are equal and all its off-diagonal entries are equal. Thus
	\[
	\overline M=aI+b(I-J)
	\]
	for some $a,b\in\C$. Since $M\one =0$, we also have $\overline M \one =0$, and it follows that $a=0$. Finally, taking traces and using that the trace of a matrix is invariant under conjugation, we obtain $\tr M =b(d-1)$.
\end{proof}

For a doubly stochastic matrix $P$, we have $PJ=JP=J$. Moreover, every permutation matrix $\Pi$ satisfies $\Pi J=J\Pi=J$. As a consequence, centering is compatible with relabelling in the following sense.

\begin{lemma}\label{lem:centered-product}
	Let $P_1,\ldots,P_L$ be doubly stochastic matrices, let
	$\Pi_1,\ldots,\Pi_L$ be permutation matrices, and set
	$Q_t=P_t-J$. Then
	\begin{equation}\label{eq:centered-product}
		\Pi_1P_1\Pi_2P_2\cdots\Pi_LP_L-J
		=
		\Pi_1Q_1\Pi_2Q_2\cdots\Pi_LQ_L.
	\end{equation}
\end{lemma}

\begin{proof}
	Let $P_1$ and $P_2$ denote two arbitrary doubly stochastic matrices, and write $Q_t=P_t-J$ for $t=1,2$. We have
	\[
	Q_1 \Pi Q_2 = (P_1-J) \Pi (P_2 - J) = P_1 \Pi P_2 - P_1 \Pi J - J\Pi P_2 + J \Pi J.
	\]
	Since $J\Pi=\Pi J=J$ and $JP_t=P_tJ=J$, it follows that
	\[Q_1 \Pi Q_2 = P_1 \Pi P_2 - J - J + J = P_1 \Pi P_2 - J.\]
	Multiplying by a permutation matrix $\Pi'$ on the left, and using $\Pi' J=J$ proves the case $L=2$. Noting that $(P_1\Pi P_2)$ is doubly stochastic, iterating this identity proves the claim.
\end{proof}

We now prove \cref{thm:exact-covariance}. The argument is based on repeatedly applying the permutation isotropy identity (\ref{eq:isotropy}), starting from the last layer of the product.

\begin{proof}[Proof of \cref{thm:exact-covariance}]
	Let $D_L=R_L-J$. By \cref{lem:centered-product},
	\[
	D_L
	=
	\Pi_1Q_1\cdots\Pi_LQ_L.
	\]
	We first consider the case $L=1$. We have
	\[
	D_1D_1^\mathsf{T}
	=
	\Pi_1Q_1Q_1^\mathsf{T}\Pi_1^\mathsf{T}
	\]
	Since $Q_1^\mathsf{T}\one=0$, the matrix $Q_1Q_1^\mathsf{T}$ satisfies
	$Q_1Q_1^\mathsf{T}\one=0$. Moreover,
	$\tr(Q_1Q_1^\mathsf{T})=\|Q_1\|_F^2=(d-1)\alpha(P_1)$. Therefore, by \cref{lem:permutation-isotropy},
	\[
	\E\left(\Pi_1Q_1Q_1^\mathsf{T}\Pi_1^\mathsf{T}\right)
	=
	\alpha(P_1)(I-J).
	\]
	For larger $L$, the result follows from iterating this process and using that $\Pi_t$ are independent uniform permutations. More precisely, condition on $\Pi_1,\ldots,\Pi_{L-1}$ and notice that the previous argument gives
	\begin{align*}
	\E\left (D_LD_L^\mathsf{T} \mid \Pi_1,\ldots,\Pi_{L-1}\right )
	&=
	\alpha(P_L)
	\Pi_1Q_1\cdots\Pi_{L-1}
	Q_{L-1}(I-J)Q_{L-1}^\mathsf{T}
	\Pi_{L-1}^\mathsf{T}\cdots
	Q_1^\mathsf{T}\Pi_1^\mathsf{T}.
	\\ 
	&=\alpha(P_L)
	\Pi_1Q_1\cdots\Pi_{L-1}
	Q_{L-1}Q_{L-1}^\mathsf{T}
	\Pi_{L-1}^\mathsf{T}\cdots
	Q_1^\mathsf{T}\Pi_1^\mathsf{T}.
	\end{align*}
	Thus averaging over the last permutation removes the last layer and contributes exactly the scalar factor $\alpha(P_L)$. Iterating this procedure for $\Pi_{L-1},\ldots,\Pi_1$ yields
	\[
	\E(D_LD_L^\mathsf{T})
	=
	\left(\prod_{t=1}^L\alpha(P_t)\right)(I-J),
	\]
	which proves the theorem.
\end{proof}

We now derive \cref{cor:exact-l2} as an immediate application of \cref{thm:exact-covariance}. Note that, for any probability row vector $\nu$ on $[d]$, the $\chi^2$-divergence from the uniform distribution $u$ satisfies
\begin{equation}\label{eq:chi-square}
	\chi^2(\nu\|u)
	=
	\sum_{j=1}^d\frac{(\nu_j-u_j)^2}{u_j}
	=
	d\sum_{j=1}^d(\nu_j-u_j)^2
	=
	d\|\nu-u\|_2^2,
\end{equation}
since $u_j=1/d$ for every $j$. Thus, in our setting, controlling the $\chi^2$-divergence of a row of $R_L$ from $u$ is equivalent to controlling its Euclidean distance.

\begin{proof}[Proof of \cref{cor:exact-l2}]
	Taking traces in \eqref{eq:left-covariance} gives
	\[
	\E\| R_L-J\|_F^2
	=
	\left(\prod_{t=1}^L\alpha(P_t)\right)\tr(I-J).
	\]
	Since $\tr(I-J)=(d-1)$, this proves \eqref{eq:exact-frobenius}. For the second assertion, set $x=\mu-u$. Note that $xJ=0$ as both $\mu$ and $u$ are probability measures. Moreover, since $u$ is stationary for every factor of $R_L$,
	\[
	\mu R_L-u = \mu R_L-u R_L=x R_L = x( R_L-J).
	\]
	Therefore, we have
	\[ \|\mu R_L-u\|_2^2 = x (R_L-J)(R_L-J)^\T x^\T.\]
	Taking expectations and applying \cref{thm:exact-covariance},
	\[
	\E\|\mu R_L-u\|_2^2
	=
	x\,\E\bigl[(R_L-J)(R_L-J)^\mathsf{T}\bigr]x^\mathsf{T}
	=
	\left(\prod_{t=1}^L\alpha(P_t)\right)
	x(I-J)x^\mathsf{T}.
	\]
	Finally, since $xJ=0$, we have $x(I-J)x^\mathsf{T}=xx^\mathsf{T}=\|x\|_2^2$. Multiplying by $d$ and using \eqref{eq:chi-square} proves the result.
\end{proof}

\subsection{The regular case}

We now specialize the preceding identities to the regular kernels considered in earlier sections. Recall that a matrix $P\in \mathbb{R}^{d\times d}$ is an $r$-regular kernel when it has exactly $r$ nonzero entries in each row and column, each equal to $1/r$. In this case, the contraction coefficient $\alpha(P)$ can be computed explicitly.

\begin{lemma}\label{lem:regular-alpha}
	Let $P \in \mathbb{R}^{d\times d}$ be an $r$-regular kernel. Then
	\begin{equation}\label{eq:regular-alpha}
		\alpha(P)
		=
		\frac{d-r}{r(d-1)}.
	\end{equation}
\end{lemma}

\begin{proof}
	Since every row of $P$ contains exactly $r$ entries equal to $1/r$, we have $\|P\|_F^2=d/r$. Moreover, since $P$ is stochastic, 
	\[ \langle P,J\rangle_F= \tr(PJ^\mathsf{T}) = \tr(J)= 1.\] 
	We also have $\|J\|_F^2=1$. Therefore,
	\[
	\|P-J\|_F^2
	=
	\|P\|_F^2-2\langle P,J\rangle_F+\|J\|_F^2
	=
	\frac{d}{r}-1.
	\]
	Dividing by $d-1$ gives the formula for $\alpha(P)$.
\end{proof}

We now return to the regular setting. Suppose that $P_t$ is $r_t$-regular
for $t=1,\ldots,L$, and set
\[
B_L=\prod_{t=1}^L r_t.
\]
For the product $R_L$ defined in \eqref{eq:regular-product},
\cref{cor:exact-l2,lem:regular-alpha} give explicit formulas for its
Frobenius distance from $J$ and for the $\chi^2$-divergence of each row
from equilibrium.

\begin{corollary}\label{cor:regular-l2}
	Under the preceding assumptions, for every $i\in[d]$,
	\begin{equation}\label{eq:regular-l2}
		\E\|R_L-J\|_F^2
		=
		\E\chi^2\bigl((R_L)_i\|u\bigr)
		=
		\frac{d-1}{B_L}
		\prod_{t=1}^L\frac{d-r_t}{d-1}.
	\end{equation}
\end{corollary}

\begin{proof}
	First, Corollary \ref{cor:exact-l2} and \cref{lem:regular-alpha} give
	\[
	\E\|R_L-J\|_F^2
	=
	(d-1)\prod_{t=1}^L\frac{d-r_t}{r_t(d-1)}
	=
	\frac{d-1}{B_L}
	\prod_{t=1}^L\frac{d-r_t}{d-1}.
	\]
	Similarly, applying Corollary \ref{cor:exact-l2} with $\mu=e_j^\mathsf{T}$ and using $\chi^2(e_j^\mathsf{T}\|u)=d-1$ gives the same formula for $\E\chi^2((R_L)_j\|u)$.
\end{proof}

\section{Total-variation profile}\label{sec:tv-profile}

We now turn to the total-variation profile in Corollary~\ref{thm:tv-profile}. The key point is that the path-multiplicity representation expresses the distance of a row from equilibrium directly in terms of the empirical multiplicity measure. Indeed, for every $i,j\in[d]$,
\[
(R_L)_{ij}=\frac{N_L(i,j)}{B_L}.
\]
Since $B_L=\lambda_Ld$, it follows that
\begin{equation}\label{eq:tv-multiplicity}
		\bigl\|(R_L)_i-u\bigr\|_{\TV}
		=
		\frac{1}{2}\sum_{j=1}^d
		\left|
		\frac{N_L(i,j)}{B_L}-\frac{1}{d}
		\right|
		=
		\frac{1}{2\lambda_L}
		\int |k-\lambda_L|\,\dd\nu_L^{(i)}(k).
\end{equation}
Thus, the total-variation distance is a Lipschitz functional of the empirical multiplicity measure. Consequently, Theorem~\ref{cor:wasserstein-poisson} already yields its limiting profile. For the quantitative estimate, we instead use Theorem~\ref{cor:poisson-berry-esseen}, which avoids the logarithmic loss in the Wasserstein bound.

\begin{proof}[Proof of Corollary~\ref{thm:tv-profile}]
	Fix $i \in [d]$. We first consider the precritical regime where $\lambda_L$ is small. Since the path multiplicities $N_L(i,j)$ sum to $B_L$, $(R_L)_i$ is supported on at most $B_L$ vertices. Equivalently, at least $d-B_L$ entries of $(R_L)_i$ are equal to zero. Therefore, for every realization,
	\[
	\bigl\|(R_L)_i-u\bigr\|_{\TV}
	\geq
	1-\frac{B_L}{d}
	=
	1-\lambda_L
	\]
	whenever $\lambda_L\leq1$, proving (i).
	
	Since $\nu_L^{(i)}$ has mean $\lambda_L$, by
	\eqref{eq:tv-multiplicity} we have
	\[
		\bigl\|(R_L)_i-u\bigr\|_{\TV}
		=
		\frac{1}{2\lambda_L}
		\int |k-\lambda_L|\,\dd\nu_L^{(i)}(k)
		=
		\frac{1}{\lambda_L}
		\int(\lambda_L-k)_+\,\dd\nu_L^{(i)}(k).
	\]
	Similarly, for $Z_{\lambda_L}\sim\Pois(\lambda_L)$,
	\[
	\Psi(\lambda_L)
	=
	\frac{1}{\lambda_L}
	\E(\lambda_L-Z_{\lambda_L})_+.
	\]
	Using the identity
	\[
	(\lambda-x)_+
	=
	\int_0^\lambda \one_{\{x<t\}}\,\dd t,
	\]
	we obtain
\[
\begin{aligned}
	\left|
	\bigl\|(R_L)_i-u\bigr\|_{\TV}
	-
	\Psi(\lambda_L)
	\right|
	&\leq
	\frac{1}{\lambda_L}
	\int_0^{\lambda_L}
	\left|
	\nu_L^{(i)}([0,t))
	-
	\Pp(Z_{\lambda_L}<t)
	\right|\, dt\\
	&\leq
	\sup_{m\geq0}
	\left|
	\nu_L^{(i)}([0,m])
	-
	\Pp(Z_{\lambda_L}\leq m)
	\right|,
\end{aligned}
\]
	where the last inequality uses that both measures are supported on the nonnegative integers. Taking expectations, the quantitative estimate now follows from
	Theorem~\ref{cor:poisson-berry-esseen}. If $\lambda_L\to\lambda\in(0,\infty)$ and $S_L\to0$, then
	\eqref{eq:tv-profile-rate-main}, together with the continuity of $\Psi$,
	gives the claimed convergence in $L^1$.
	
	Finally, the standard inequality between total variation and $\chi^2$-divergence gives
	\[
	\bigl\|(R_L)_i-u\bigr\|_{\TV}
	\leq
	\frac{1}{2}
	\sqrt{\chi^2\bigl((R_L)_i\|u\bigr)}.
	\]
	Hence, by Jensen's inequality and Corollary~\ref{cor:regular-l2},
	\[
	\E\bigl\|(R_L)_i-u\bigr\|_{\TV}
	\leq
	\frac{1}{2}
	\sqrt{
		\E\chi^2\bigl((R_L)_i\|u\bigr)
	}
	\leq
	\frac{1}{2}\sqrt{\frac{d-1}{B_L}}	\leq
	\frac{1}{2\sqrt{\lambda_L}}.
	\]
When $\lambda_L\to \infty$, this provides convergence in $L^1$.
\end{proof}

\section{Applications}\label{sec:applications}

\subsection{Synchronous random-group averaging}

We return to the synchronous opinion averaging model introduced in
\cref{subsec:applications}, in which the agents are independently
reshuffled at each round into rooms of size $b$ and replace their
opinions by the average opinion in their room. We first prove
Theorem~\ref{thm:synchronous-averaging} and
Corollary~\ref{cor:synchronous-cutoff}, and then discuss some concrete
choices of the room size.

Recall that $\Psi(\lambda)=\frac{1}{2\lambda} \E |Z_\lambda - \lambda|$, where $Z_\lambda \sim \Pois(\lambda)$, and that we measure worst-case convergence to consensus by
\[
	\mathcal C_d(L)
	=
	\sup_{\mu\in\mathcal P([d])}
	\E\|\mu R_L-u\|_{\TV},
\]
where the supremum is over all probability vectors $\mu$ on $[d]$.

\begin{proof}[Proof of Theorem \ref{thm:synchronous-averaging}]
	The proof follows from a direct application of Theorem \ref{cor:wasserstein-poisson} and Corollary \ref{thm:tv-profile} once we see that the model falls under our framework with parameters
	\begin{equation}\label{eq:parameters-model}
	B_L=b^L,
	\qquad
	S_L=\frac{bL}{d},
	\qquad
	\lambda_L=\frac{b^L}{d}.
	\end{equation}
	Indeed, the Poisson convergence follows directly from (\ref{eq:wasserstein-bound-main}). On the other hand, since the map $\mu \mapsto \mathbb{E} \|\mu R_L - u\|_{\TV}$ is convex on the simplex $\Delta_d=\{x \in [0,1]^d \colon \sum_{i=1}^d x_i = 1\}$, its supremum is attained at an extreme point. Consequently,
	\[ \mathcal C_d(L) = \max_{i=1,\ldots, d} \mathbb{E}\|(R_L)_i - u \|_{\TV},\]
	and Corollary \ref{thm:tv-profile} provides uniform convergence.
	
	It remains to verify that the parameters of the model are indeed given by
	\eqref{eq:parameters-model}. Let $G$ be the graph consisting of $d/b$
	disjoint complete components of size $b$, with self-loops, and let $P$
	be its transition matrix. Thus $P$ is $b$-regular, and applying
	$P$ averages the opinions within each room. Since at every round the agents are reshuffled independently and uniformly, the successive changes of room assignments are independent
	uniform permutations. Hence, up to a final permutation of the agents,
	the dynamics is exactly of the form \eqref{eq:regular-product}, with
	\[
	P_1=\cdots=P_L=P,
	\qquad
	r_1=\cdots=r_L=b.
	\]
	Consequently, $B_L=b^L$, $\lambda_L=b^L/d$, and
	\[ S_L=\frac{1}{d}\sum_{\ell=1}^L b = \frac{bL}{d}.
	\]
	The final permutation merely relabels the agents and does not affect
	either $\nu_L^{(i)}$ or the total-variation distance from $u$.
\end{proof}

\begin{proof}[Proof of Corollary \ref{cor:synchronous-cutoff}]
	Set $n=\log_b d$ and take $L=n+c$. Then
	\[
	\lambda_L=\frac{b^{n+c}}{b^n}=b^c,
	\qquad
	S_L=\frac{b(n+c)}{b^n}\longrightarrow0.
	\]
	From
	Theorem~\ref{thm:synchronous-averaging} we deduce that
	\[ \mathcal C_{d}(n+c) \longrightarrow \Psi(b^c).\] 
	Since
	\[
	\Psi(b^c)\longrightarrow1
	\quad\text{as }c\to-\infty,
	\qquad
	\Psi(b^c)\longrightarrow0
	\quad\text{as }c\to+\infty,
	\]
	the transition to consensus takes place in a
	bounded window around $n=\log_b d$.
\end{proof}

When $b=2$, each round is a uniform perfect matching of the agents, with
all matched pairs averaging simultaneously. For $d=2^n$, the consensus
profile is
\[
\mathcal C_d(n+c)\longrightarrow \Psi(2^c),
\qquad c\in\mathbb Z.
\]
In particular, the cutoff occurs at $n=\log_2 d$ with a bounded window,
and at the critical round $L=n$ one has
\[
\mathcal C_d(n)\longrightarrow \Psi(1)=e^{-1}.
\]
At the same time,
\[
\nu_n^{(i)}(\{k\})\longrightarrow \frac{e^{-1}}{k!},
\qquad k\geq0,
\]
so the influence multiplicities of a fixed agent have an asymptotic
$\Pois(1)$ profile.

If instead $d=m^2$ and $b=m=\sqrt d$, then the transition takes place
in only a few rounds:
\[
\mathcal C_d(1)\longrightarrow1,
\qquad
\mathcal C_d(2)\longrightarrow e^{-1},
\qquad
\mathcal C_d(3)\longrightarrow0.
\]
Thus the critical Poisson profile appears after two rounds, while
consensus distance vanishes after one additional round.

Finally, we prove Theorem~\ref{thm:heterogeneous-averaging}.

\begin{proof}[Proof of Theorem~\ref{thm:heterogeneous-averaging}]
	At round $t$, after fixing an ordering of the rooms, the corresponding
	averaging kernel can be written as
	\[
	P_t
	=
	\operatorname{diag}\left(
	\frac{1}{b_{t,1}}\one\one^\mathsf T,
	\ldots,
	\frac{1}{b_{t,q_t}}\one\one^\mathsf T
	\right),
	\]
	where each block has the appropriate dimension. Thus $P_t$ is an
	orthogonal projection of rank $q_t$, and therefore
	\[
	\|P_t\|_F^2=q_t.
	\]
	Since $P_t$ is doubly stochastic,
	\[
	\|P_t-J\|_F^2
	=
	\|P_t\|_F^2-1
	=
	q_t-1.
	\]
	Hence
	\[
	\alpha(P_t)
	=
	\frac{q_t-1}{d-1}
	\leq
	\frac{q_t}{d}.
	\]
	By Corollary~\ref{cor:exact-l2}, for every probability vector $\mu$,
	\[
	\E\chi^2(\mu R_L\|u)
	=
	\chi^2(\mu\|u)
	\prod_{t=1}^L\alpha(P_t) \leq d\prod_{t=1}^L\frac{q_t}{d},
	\]
	where we have used that $\chi^2(\mu\|u)\leq d$. Using \eqref{eq:standard_divergence} and Jensen's inequality,
	\[
	\E\|\mu R_L-u\|_{\TV}
	\leq
	\frac12
	\sqrt{
		d\prod_{t=1}^L\frac{q_t}{d}
	}.
	\]
	Taking the supremum over all measures $\mu$ on $[d]$ gives the result.
\end{proof}

\section*{Acknowledgments}

This work was supported in part by the Center of Interdisciplinary Data Science and Artificial
Intelligence (CIDSAI) at New York University Abu Dhabi, funded by Tamkeen under the NYUAD Research Institute Award CG016.

\section*{Statement on the use of AI}
The mathematical ideas and original arguments in this work were developed by the author. During the preparation of the manuscript, ChatGPT (GPT-6) was used to refine and simplify proofs, as well as to obtain explicit quantitative constants. The aim was to preserve the original mathematical intuition while presenting the arguments in the clearest and simplest possible form. ChatGPT was also used to assist with literature review, manuscript editing, and the preparation of figures. All mathematical arguments were independently verified by the author, who takes full responsibility for the manuscript.

\end{document}